\documentclass{article}

\usepackage{amsmath}
\usepackage{amssymb}
\usepackage{graphicx}
\usepackage{bm}
\usepackage{mathdots}
\usepackage{booktabs}
\usepackage{rotating}
\usepackage{listings}
\usepackage[ruled,linesnumbered]{algorithm2e}
\usepackage[a4paper,left=2.8cm,right=2.8cm,top=2.5cm,bottom=2.5cm]{geometry}
\usepackage{fancyhdr}
\usepackage[framemethod=tikz]{mdframed}
\newmdtheoremenv[%
backgroundcolor=green!10,%
outerlinecolor=black,%
leftmargin=0,%
rightmargin=0,
innertopmargin =3pt,%
innerleftmargin = 5pt,
innerrightmargin = 5pt,
splittopskip = \topskip,%
skipabove = \baselineskip,%
skipbelow = \baselineskip,%
roundcorner=5, ntheorem]
{theorem}{Theorem}[section]

\newtheorem{lemma}{Lemma}[section]

\makeatletter % `@' now normal "letter"
\@addtoreset{equation}{section}
\makeatother  % `@' is restored as "non-letter"

\newtheorem{remark}{Remark}[section]

\newenvironment{proof}{{\noindent\it Proof.}\quad}{\hfill $\square$\\}

\usepackage{url}
\usepackage{mathtools}
\usepackage{lipsum}

\newcommand{\Lt}{L^2}

\newcommand{\St}{\mathbb{S}^2}
\newcommand{\Pn}{\mathbb{P}_n}
\newcommand{\RR}{\mathbb{R}}

\newcommand{\x}{\bm{x}}
\newcommand{\y}{\bm{y}}

\newcommand{\X}{\mathcal{X}}

\begin{document}
% Stable image reconstruction via the springback-penalized model
\title{Spherical configurations and quadrature methods for integral equations of the second kind}

\author{Congpei An\footnotemark[1]
       \quad\quad Hao-Ning Wu\footnotemark[2]}

\renewcommand{\thefootnote}{\fnsymbol{footnote}}
\footnotetext[1]{Department of Mathematics, The University of Hong Kong, Hong Kong, China (andbachcp@gmail.com)}
\footnotetext[2]{Department of Mathematics, University of Georgia, Athens, GA 30602, USA (hnwu@uga.edu)}

% \date{}
%\shorttitle{SHORT TITLE}
%\shortauthor{F.~FIRSTA AND S.~SECONDA}
% \usepackage{keywords}

\maketitle

\begin{abstract}
In this paper, we propose and analyze a product integration method for the second-kind integral equation with weakly singular and continuous kernels on the unit sphere $\mathbb{S}^2$. We employ quadrature rules that satisfy the Marcinkiewicz--Zygmund property to construct hyperinterpolation for approximating the product of the continuous kernel and the solution, in terms of spherical harmonics. By leveraging this property, we significantly expand the family of candidate quadrature rules and establish a connection between the geometrical information of the quadrature points and the error analysis of the method. We then utilize product integral rules to evaluate the singular integral with the integrand being the product of the singular kernel and each spherical harmonic.  We derive a practical $L^{\infty}$ error bound, which consists of two terms: one controlled by the best approximation of the product of the continuous kernel and the solution, and the other characterized by the Marcinkiewicz--Zygmund property and the best approximation polynomial of this product. Numerical examples validate our numerical analysis.
\end{abstract}

\textbf{Keywords: }{integral equation, product integration, Nystr\"{o}m method, quadrature, hyperinterpolation, point distribution}

\textbf{AMS subject classifications.} 45B05, 65R20, 65D32, 41A10

\section{Introduction}
We consider the Fredholm integral equation of the second kind
\begin{equation}\label{equ:equation}
\varphi(\x) -\int_{\St}h(|\x-\y|)K(\x,\y)\varphi(\y){ d}\omega(\y) = f(\x)
\end{equation}
on the 2-sphere $\St:=\{\x\in\mathbb{R}^3:\|\x\|_2=1\}$ with surface measure ${ d}\omega$, where $|\x-\y|:=\sqrt{2(1-\x\cdot \y)}$ denotes the Euclidean distance between points $\x$ and $\y$ on $\St$. The inhomogeneous term $f$, the kernel $K$, and the solution $\varphi$ are assumed to be continuous, whereas the weight function $h:(0,\infty)\rightarrow\RR$ is allowed to be weakly singular. It is assumed throughout this paper that the homogeneous equation corresponding to \eqref{equ:equation} has no non-trivial solution; then as a consequence of the classic Riesz theory \cite{riesz1918lineare}, the inhomogeneous equation \eqref{equ:equation} has a unique solution continuously depending on $f$. In the past decades, numerically solving \eqref{equ:equation} has attracted much attention, and many important works have emerged, see, for example, \cite{MR1464941,MR3397293,MR2986407,MR2371991,graham2002fully,MR2457245,tran2009boundary} and references therein.

In this paper we are concerned with a class of quadrature-type methods for the numerical solution of the equation \eqref{equ:equation}. More precisely, we are interested in computing the approximate solution $\varphi_{\gamma}$ that satisfies an equation of the form
\begin{equation}\label{equ:nystromform}
\varphi_{\gamma}(\x) -\sum_{j=1}^m W_j(\x)K\left(\x,\x_j\right)\varphi_{\gamma}\left(\x_j\right) = f(\x),\quad \x\in\St,
\end{equation}
where $\varphi_{\gamma}$ is the computed solution with index $\gamma$ explained in Section \ref{sec:proposedscheme}, $\x_1,\ldots,\x_m$ are distinct points on $\St$, and $W_1,\ldots,W_m$ are continuous scalar-valued functions on $\St$. For a fixed value of $\x$, we may interpret $\{\x_j\}$ and $\{W_j\}$ as the points and weights in a certain quadrature rule
\begin{equation}\label{equ:prodquad}
\sum_{j=1}^mW_j(\x) g\left(\x_j\right)\approx \int_{\St} h(|\x-\y|)g(\y){ d}\omega(\y)
\end{equation}
for some continuous function $g$, in which the possibly singular factor $h(|\x-\y|)$ has been incorporated into the weights $W_j(\x)$ implicitly. In particular, if $h(|\x-\y|)\equiv 1$, then weights $W_j$ become constants $w_j$ and the quadrature \eqref{equ:prodquad} reduces to the classic quadrature form with scalar weights $w_j$:
\begin{equation}\label{equ:quad}
\sum_{j=1}^mw_j g\left(\x_j\right)\approx \int_{\St} g(\y){ d}\omega(\y).
\end{equation}

The goal of this paper is to propose an intrinsic scheme for the integral equation \eqref{equ:equation} on $\St$ by providing an explicit construction of $W_j(\x)$ and to establish a unified convergence theory for approximations of the form \eqref{equ:nystromform}, incorporating with the geometry (distribution information) of $\X_m:=\{\x_1,\ldots,\x_m\}\subset{\mathbb{S}^2}$. This convergence theory not only addresses convergence itself, which can be achieved using Anselone’s collectively compact operator approximation theory \cite{MR0443383} with Sloan’s modification for the singular term \cite{MR638450}, but also the \emph{rate} of convergence, derived with the aid of the authors’ recent work on hyperinterpolation \cite{an2022quadrature,an2024bypassing,an2024hyperinterpolation}. It should also be noted that Sloan’s modification in \cite{MR638450} for singular kernels was proposed for equations on the 1D interval, while we extend it to the 2D sphere in this paper. Based on the relation between spherical configuration and the error bound of hyperinterpolation, the obtained rate of convergence can be further refined to \emph{a priori} rates, provided that the geometrical information of the distribution of $\mathcal{X}_m$ is known.

The study of the distribution of $\X_m$ on the sphere has numerous applications, including climate modeling, global approximation in geophysics, and virus modeling in bioengineering, given that both the Earth and cells can be modeled as spheres. In classical numerical analysis, these points are often required to satisfy some exactness assumptions, that is, the quadrature rule using these points are exact for polynomials of certain degrees. However, in more realistic applications, these points may not fulfill this desired assumption. To cover both cases, in this paper, we assume $\X_m$ to satisfy the Marcinkiewicz--Zygmund property. Then many spherical configurations are feasible for the construction of quadrature rule \eqref{equ:quad} as well as the weights $W_j(\x)$, such as randomly distributed points \cite{MR2475947}, equal area points \cite{MR1306011}, minimal energy points \cite{MR1845243}, maximal determinant points (Fekete points) \cite{MR1845243}, spherical $t$-designs \cite{delsarte1991geometriae}, and so on. Even for a set of scattered points, we can extract its geometrical information, namely mesh norm, for the Marcinkiewicz--Zygmund property of polynomials of certain degrees, see, for example, \cite{filbir2011marcinkiewicz,mhaskar2001spherical}.

In Section \ref{sec:proposedscheme}, we propose our scheme and justify the underlying assumptions. Section \ref{sec:error} establishes the convergence theory. Section \ref{sec:implementation} discusses three specific singular kernels and addresses techniques for computing with them. Numerical experiments are presented in Section \ref{sec:numerical}. Finally, Section \ref{sec:discussion} provides concluding discussions.

\section{Proposed scheme}\label{sec:proposedscheme}

We denote by $\{Y_{\ell,k}:k=1,\ldots,2\ell+1,~\ell=0,1,2,\ldots\}$ the $L^2$-orthogonal system of real spherical harmonics of degree $\ell$ and order $k$ \cite{MR0199449}, and let $\mathbb{P}_n$ be the polynomial space of degree $n$ that is spanned by spherical harmonics of degree at most $n$. Let $C(\St)$ be the space of continuous functions on $\St$ equipped with the uniform norm 
$\|g\|_{L^\infty}:=\max_{\x\in\St}|g(\x)|$.
\subsection{Hyperinterpolation}
For a continuous function $g\in C(\St)$, its hyperinterpolation of degree $n$ is defined as
\begin{equation}\label{equ:hyper}
\mathcal{L}_ng:= \sum_{\ell=0}^n\sum_{k = 1}^{2\ell+1} \left\langle g,Y_{\ell,k}\right\rangle_m Y_{\ell,k},
\end{equation}
where 
$$\langle g,Y_{\ell,k}\rangle_m:=\sum_{j=1}^m w_j g(\x_j)Y_{\ell,k}(\x_j)$$
is numerically evaluated by the quadrature rule \eqref{equ:quad} with quadrature points $\X_m$ and positive quadrature weights $w_j>0$, $j=1,\ldots,m$. The approximation scheme is due to Sloan \cite{sloan1995polynomial}, in which the quadrature rule \eqref{equ:quad} is assumed to be exact for polynomials (spherical harmonics) of degree at most $2n$, that is, 
\begin{equation}\label{equ:exact}
\sum_{j=1}^mw_jg(\x_j) = \int_{\St}\chi(\x)d\omega(\x)\quad \forall \chi\in\mathbb{P}_{2n}.
\end{equation}
Regarding this very restrictive nature of quadrature exactness, it is impractical and sometimes impossible to obtain data on the desired quadrature points in practice. In a series of recent works \cite{an2022quadrature,an2024bypassing,an2024hyperinterpolation} by the authors, this exactness assumption can be relaxed to the following Marcinkiewicz--Zygmund property: We assume that there exists an $\eta\in[0,1)$, which is independent of $n$ and $\chi$, such that
\begin{equation}\label{equ:MZ}
(1-\eta) \int_{\mathbb{S}^2} \chi^2 d \omega \leq \sum_{j=1}^m w_j \chi\left(\x_j\right)^2 \leq(1+\eta) \int_{\mathbb{S}^2} \chi^2 d \omega \quad \forall \chi \in \mathbb{P}_n.
\end{equation}
If the quadrature exactness \eqref{equ:exact} is satisfied, then $\eta=0$. Then the construction of hyperinterpolation is feasible with many more quadrature rules outside the traditional candidates. This property \eqref{equ:MZ} can be regarded as the Marcinkiewicz--Zygmund inequality \cite{MR3746524,filbir2011marcinkiewicz,MR2475947,Marcinkiewicz1937,MR2086950,mhaskar2001spherical,MR1863015}  applied to polynomials $\chi^2$ of degree at most $2n$ with $\chi\in\mathbb{P}_n$.

In our previous work \cite{an2024bypassing} on hyperinterpolation (referred to \emph{unfettered hyperinterpolation} there), if the $m$-point positive-weight quadrature rule \eqref{equ:quad} satisfies the Marcinkiewicz--Zygmund property \eqref{equ:MZ} with $\eta\in[0,1)$, then
\begin{equation}\label{equ:hyperstability}
\|\mathcal{L}_ng\|_{L^2}\leq\sqrt{1+\eta}\left(\sum_{j=1}^mw_j\right)^{1/2}\|g\|_{\infty}
\end{equation}
and
 \begin{equation}\label{equ:hypererror}
\|\mathcal{L}_ng-g\|_{\Lt}\leq \left(\sqrt{1+\eta}\left(\sum_{j=1}^mw_j\right)^{1/2}+\lvert\St\rvert^{1/2}\right)E_n(g)+\sqrt{\eta^2+4\eta}\|\chi^*\|_{\Lt},
\end{equation} 
where $E_n(g)$ denotes the best uniform error of $g$ by a polynomial in $\mathbb{P}_n(\St)$ and $\chi^*\in\mathbb{P}_n(\St)$ denotes the best approximation polynomial of $g$ in $\mathbb{P}_n(\St)$ in the sense of $\|g-\chi^*\|_{L^\infty}=E_n(g)$.

\subsection{Numerical scheme}

Our scheme proposes the replacement of the term $K(\x,\y)\varphi(\y)$ with its \emph{hyperinterpolation} of degree $n$, forming the basis of our approach. Therefore, for the integral operator in the equation \eqref{equ:equation}, we have its approximation
\begin{equation}\label{equ:derivation}\begin{split}
\int_{\St}h(|\x-\y|)\mathcal{L}_n&\left(K(\x,\y)\varphi(\y)\right){ d}\omega(\y)\\
& =  \int_{\St}h(|\x-\y|)\left(\sum_{\ell=0}^n\sum_{k=1}^{2\ell+1}\left\langle K(\x,\cdot)\varphi,Y_{\ell,k}\right\rangle_mY_{\ell,k}(\y)\right){ d}\omega(\y)\\
& = \sum_{\ell=0}^n\sum_{k=1}^{2\ell+1} \left(\int_{\St}h(|\x-\y|) Y_{\ell,k}(\y){ d}\omega(\y)\right)\left\langle K(\x,\cdot)\varphi,Y_{\ell,k}\right\rangle_m\\
& = \sum_{j=1}^mw_j\left(\sum_{\ell=0}^n\sum_{k=1}^{2\ell+1} \left(\int_{\St}h(|\x-\y|) Y_{\ell,k}(\y){ d}\omega(\y)\right)Y_{\ell,k}(\x_j)\right) K(\x,\x_j)\varphi(\x_j)\\
&=:\sum_{j=1}^mW_j(\x)K(\x,\x_j)\varphi(\x_j),
\end{split}\end{equation}
where 
\begin{equation}\label{equ:Wj}
W_j(\x):= w_j\left(\sum_{\ell=0}^n\sum_{k=1}^{2\ell+1} \left(\int_{\St}h(|\x-\y|) Y_{\ell,k}(\y){ d}\omega(\y)\right)Y_{\ell,k}(\x_j)\right).
\end{equation}

Together with the particular weights \eqref{equ:Wj} of our proposed method, the solution $\varphi_{\gamma}(\bm{t})$ at an arbitrary point $\bm{t}\in\St$ of approximate form \eqref{equ:nystromform} can be solved in two stages: firstly, one sets $\bm{t} = \x_j$, $j=1,\ldots,m$, then numerically solves the obtained system of linear equations
\begin{equation}\label{equ:scheme1}
\varphi_{\gamma}(\x_i) - \sum_{j=1}^mW_j(\x_i)K(\x_i,\x_j)\varphi_{\gamma}(\x_j) = f(\x_i),\quad i = 1,\ldots,m
\end{equation}
for the quantities $\varphi_{\gamma}(\x_j)$, $j=1,\ldots,m$; secondly, the value of $\varphi_{\gamma}(\bm{t})$ at any $\bm{t}\in\St$ can be evaluated by the direct usage of \eqref{equ:nystromform},
\begin{equation}\label{equ:scheme2}
\varphi_{\gamma}(\bm{t})=f(\bm{t})+\sum_{j=1}^mW_j(\bm{t})K(\bm{t},\x_j)\varphi_{\gamma}(\x_j).
\end{equation}
In our two-stage scheme \eqref{equ:scheme1} and \eqref{equ:scheme2}, the triple index $\gamma := (m,n,\eta)$ suggests that the numerical solution $\varphi_{\gamma}$ depends on the number $m$ of quadrature points, degree $n$ of hyperinterpolation, and constant $\eta$ in the Marcinkiewicz--Zygmund property.

% The numerical scheme \eqref{equ:scheme} can be written in a matrix form as 
% \begin{equation*}
% {\footnotesize \begin{bmatrix}
% \varphi_{\gamma}(\x_1)\\\varphi_{\gamma}(\x_2)\\\vdots\\ \varphi_{\gamma}(\x_m)
% \end{bmatrix}
% -\begin{bmatrix}
% W_1(\x_1)K(\x_1,\x_1) & W_2(\x_1)K(\x_1,\x_2) & \cdots & W_m(\x_1)K(\x_1,\x_m)\\
% W_1(\x_2)K(\x_2,\x_1) & W_2(\x_2)K(\x_2,\x_2) & \cdots & W_m(\x_2)K(\x_2,\x_m)\\
% \vdots                & \vdots                &        & \vdots               \\
% W_1(\x_m)K(\x_m,\x_1) & W_2(\x_m)K(\x_m,\x_2) & \cdots & W_m(\x_m)K(\x_m,\x_m)
% \end{bmatrix}\begin{bmatrix}
% \varphi_{\gamma}(\x_1)\\\varphi_{\gamma}(\x_2)\\\vdots\\ \varphi_{\gamma}(\x_m)
% \end{bmatrix} = 
% \begin{bmatrix}
% f(\x_1)\\f(\x_2)\\\vdots\\ f(\x_m)
% \end{bmatrix}.}
% \end{equation*}

\subsection{Assumptions and their justification}

Let the integral operator occurred in \eqref{equ:equation} be denoted by
$$(A\varphi)(\x):=\int_{\St}h(|\x-\y|)K(\x,\y)\varphi(\y)d\omega(\y),\quad \varphi\in C(\St),\quad \x \in \St.$$
We also denote by
$$(A_{\gamma}\varphi)(\x):=\sum_{j=1}^mW_j(\x)K(\x,\x_j)\varphi(\x_j),\quad \varphi\in C(\St),\quad \x \in \St.$$

The common theoretical setting for the convergence study is the Banach space $C(\St)$. The kernel $h$ is assumed to be weakly singular, i.e., $h$ is continuous for all $\x,\y\in \St$ with $\x\neq \y$, and there exists positive constants $M$ and $\alpha \in (0,2]$ such that 
\begin{equation}\label{equ:weaklyasingular}
|h(|\x-\y|)|\leq M|\x-\y|^{\alpha-2}.
\end{equation}
These conditions ensure, by the Arzelà--Ascoli theorem, that the integral operator $A$ is a compact integral operator on $C(\St)$. It is assumed that the homogeneous equation corresponding to \eqref{equ:equation} has no non-trivial solution, then the Riesz theory confirms that the inhomogeneous equation \eqref{equ:equation} has a unique solution continuously depending on $f$. The detailed introduction on the analysis in this common setting can be found in, e.g., \cite[Chapters 2 and 3]{MR1723850}.

In the implementation of hyperinterpolation, we assume that all quadrature weights are positive, i.e., $w_j>0$ for all $j=1,2,\ldots,m$, and they satisfy $\sum_{j=1}^mw_j<C$ for some constant $C$. Additionally, we assume the quadrature rule \eqref{equ:quad} to satisfy the Marcinkiewicz--Zygmund property \eqref{equ:MZ}. Let $\Gamma$ denote the set of triples
\begin{equation*}\begin{split}
\Gamma:=\{&\gamma=(m,n,\eta):~\text{the $m$-point quadrature \eqref{equ:quad} satisfies the Marcinkiewicz--Zygmund}\\&\text{property \eqref{equ:MZ} with degree $n$ and constant $\eta$}\},
\end{split}\end{equation*}
and we assume $\gamma\in\Gamma$.

In this paper, the weakly singular assumption \eqref{equ:weaklyasingular} of $h$ is strengthened to 
\begin{equation}\label{equ:h}
h(2^{1/2}(1-t)^{1/2})\in L^1(-1,1)\cap L^2(-1,1)
\end{equation}
for $t \in (-1,1)$:
\begin{itemize}
\item  The $L^1$ assumption is indeed implied by weak singularity \eqref{equ:weaklyasingular} of $h$, with the aid of the Funk--Hecke formula (see, e.g., \cite[Theorem 2.22]{MR2934227}).
\begin{lemma}[Funk--Hecke formula]
Let $g\in L^1(-1,1)$ and $\x \in \St$. Then
$$\int_{\St}g(\x\cdot\y) Y_{\ell,k}(\y){ d}\omega(\y) = \mu_{\ell} Y_{\ell,k}(\x), 
$$
where
\begin{equation*}
\mu_{\ell}:= 2\pi \int_{-1}^1g(2^{1/2}(1-t)^{1/2})P_{\ell}(t){ d}t,
\end{equation*}
and $P_{\ell}(t)$ is the standard Legendre polynomial of degree $\ell$.
\end{lemma}
Thus, since $h(2^{1/2}(1-t)^{1/2}) = h(\sqrt{2(1-\x\cdot\y)}) = h(|\x-\y|)$, for $\alpha\in(0,2]$,
\begin{equation}\label{equ:L1assumption}\begin{split}
\int_1^1|h(2^{1/2}(1-t)^{1/2})|dt&=\frac{1}{2\pi}\int_{\St}|h(|\x-\y|)|d\omega(\y)\leq\frac{M}{2\pi}\int_{\St}|\x-\y|^{\alpha-2}d\omega(\y)\\
&\leq \frac{M}{2\pi}2^{(\alpha-2)/2}\int_{-1}^1|1-t|^{(\alpha-2)/2}dt<\infty.
\end{split}\end{equation}
\item The $L^2$ assumption is the only assumption that differs from the common setting in this paper. As we shall see from the subsequent theoretical sections, this assumption enables us to apply the $L^2$ theory of hyperinterpolation. A sufficient condition for this $L^2$ assumption, as shown in the process of \eqref{equ:L1assumption}, is $\alpha\in(0,1]$.
\end{itemize}

Besides, the $L^1$ assumption allows us to apply the Funk--Hecke formula \cite[Theorem 2.22]{MR2934227} to compute the weights \eqref{equ:Wj} efficiently by analytically evaluating the modified moments in the weights \eqref{equ:Wj} as 
\begin{equation}\label{equ:mm}
\int_{\St}h(|\x-\y|) Y_{\ell,k}(\y){ d}\omega(\y) = \mu_{\ell} Y_{\ell,k}(\x), 
\end{equation}
where
\begin{equation*}
\mu_{\ell}:= 2\pi \int_{-1}^1h(2^{1/2}(1-t)^{1/2})P_{\ell}(t){ d}t.
\end{equation*}

\section{Error analysis}\label{sec:error}

Now we present the error analysis for our two-stage scheme \eqref{equ:scheme1} and \eqref{equ:scheme2} in solving the integral equation \eqref{equ:equation}.

\begin{theorem}[Main theorem]\label{thm:main}
Let $\gamma = (m,n,\eta) \in \Gamma$ with sufficiently large $n$ and sufficiently small $\eta$. Then for the integral equation \eqref{equ:equation} with continuous kernel $K$, weakly singular kernel $h$ satisfying \eqref{equ:h}, and continuous inhomogeneous term $f$, the numerical scheme \eqref{equ:scheme1} and \eqref{equ:scheme2} computes an approximate solution $\varphi_{\gamma}$ satisfying
\begin{equation}\label{equ:stability}
\|\varphi_{\gamma}\|_{L^{\infty}}\leq C_1(m,n,\eta)\|f\|_{L^{\infty}},
\end{equation}
where $C_1(m,n,\eta)>0$ is some constant decreasing as $n$ grows or $\eta$ decreases.
Moreover, there exists $\x_0\in\St$ such that
\begin{equation}\label{equ:error}
\|\varphi_{\gamma}-\varphi\|_{L^{\infty}}\leq C_2(m,n,\eta)\left(E_n(K(\x_0,\cdot)\varphi) + \sqrt{\eta^2+4\eta}\|\chi^*\|_{L^{2}}\right),
\end{equation}
where $C_2(m,n,\eta)>0$ is some constant decreasing as $n$ grows or $\eta$ decreases, and $\chi^*$ stands for the best approximation polynomial of $K(\x_0,\cdot)\varphi(\cdot)$ in $\Pn$. In other words, $\varphi_{\gamma}$ uniformly converges to $\varphi$ as $n\rightarrow\infty$ and $\eta\rightarrow0$ simultaneously.
\end{theorem}

\begin{remark}
The constant $\eta$ is the error bound \eqref{equ:error} can be extracted as a geometrical information of $\mathcal{X}_m=\{\x_j\}_{j=1}^m$. If the quadrature rule \eqref{equ:quad} using these point satisfies the quadrature exactness assumption \eqref{equ:exact}, then $\eta=0$. For general $\mathcal{X}_m$, we define the \emph{mesh norm} of $\mathcal{X}_m\subset\mathbb{S}^{d-1}$ as
$$h_{\mathcal{X}_m} : =\arccos(\x\cdot\x_j).$$
In other words, the mesh norm can be regarded as the geodesic radius of the largest hole in the mesh $\mathcal{X}_m$. Thus, it was investigated in \cite{filbir2011marcinkiewicz,mhaskar2001spherical} that the Marcinkiewicz--Zygmund property \eqref{equ:MZ} holds if 
\begin{equation*}
n\lesssim\frac{\eta}{2h_{\mathcal{X}_{m}}}.
\end{equation*}
Moreover, this property \eqref{equ:MZ} holds even when $\mathcal{X}_m$ consists of random points. When the quadrature rule \eqref{equ:quad} is equal-weight, it was shown in \cite{MR2475947} that if an independent random sample of $m$ points drawn from the measure $d\omega$, then there exists a constant $\bar{c}:=\bar{c}(\gamma)$ such that the MZ inequality \eqref{equ:MZ} holds with probability exceeding $1-\bar{c}n^{-\gamma}$ on the condition of 
$$m\geq \bar{c} \frac{n ^{d-1}\log{n}}{\eta^2}.$$

\end{remark}

\subsection{Proof of the main theorem}

The integral equation \eqref{equ:equation} can be written as an operator equation
\begin{equation}\label{equ:A}
\varphi-A\varphi=f,
\end{equation}
and its solution is approximated, using our two-stage scheme \eqref{equ:scheme1} and \eqref{equ:scheme2}, by the solution of 
\begin{equation}\label{equ:Am}
\varphi_{\gamma} - A_{\gamma}\varphi_{\gamma}=f.
\end{equation}
Subtracting \eqref{equ:A} from \eqref{equ:Am} gives
\begin{equation}\label{equ:subtracting}
(I-A_{\gamma})(\varphi_{\gamma}-\varphi)=(A_{\gamma}-A)\varphi.
\end{equation}
The existence of $\varphi_{\gamma}$ and the error bound of $\|\varphi_{\gamma}-\varphi\|_{L^{\infty}}$ depend on the existence and boundedness of $(I-A_{\gamma})^{-1}$. This fact can be confirmed by the following lemma. Note that $\|A\|$ denotes the $L^{\infty}$ operator norm for an operator $A:C(\St)\rightarrow C(\St)$.

\begin{lemma}\label{lem:main}
Assume that operators $A_{\gamma}$ are compact and $\gamma\in\Gamma$ such that the sequence $\{A_{\gamma}\}$ satisfies
\begin{equation}\label{equ:etaA}
\|(I-A)^{-1}(A_{\gamma}-A)A_{\gamma}\|< 1,
\end{equation}
Then the inverse operators $(I-A_{\gamma})^{-1}:C(\St)\rightarrow C(\St)$ exist and are bounded by
\begin{equation}\label{equ:I-Am}
\|(I-A_{\gamma})^{-1}\|\leq \frac{1+\|(I-A)^{-1}A_{\gamma}\|}{1-\|(I-A)^{-1}(A_{\gamma}-A)A_{\gamma}\|}.
\end{equation}
For solutions of the equations
$$\varphi-A\varphi =f\quad \text{and}\quad \varphi_{\gamma}-A_{\gamma}\varphi_{\gamma}=f,$$
we have the error estimate
\begin{equation}\label{equ:lemerror}
\|\varphi_{\gamma}-\varphi\|_{L^{\infty}}\leq \frac{1+\|(I-A)^{-1}A_{\gamma}\|}{1-\|(I-A)^{-1}(A_{\gamma}-A)A_{\gamma}\|} \|(A_{\gamma}-A)\varphi\|_{L^{\infty}}.
\end{equation}
\end{lemma}

\begin{proof}
Recall that $A$ is a compact operator, and the equation \eqref{equ:A} has a unique solution $\varphi$ continuously depending on $f$. By Riesz's theorem (cf. \cite[Theorem 3.4]{MR1723850}), the inverse operator $(I-A)^{-1}:C(\St)\rightarrow C(\St)$ exists and is bounded. Then the identity
$$(I-A)^{-1} = I + (I-A)^{-1}A$$
suggests that 
$$B_{\gamma}:=I+(I-A)^{-1}A_{\gamma}$$
is an approximate inverse for $I-A_{\gamma}$. Note that 
\begin{equation*}\begin{split}
(I-A)B_{\gamma}(I-A_{\gamma})
& = (I-A)(I+(I-A)^{-1}A_{\gamma} -A_{\gamma} -(I-A)^{-1}A_{\gamma}A_{\gamma})\\
& = (I-A) + A_{\gamma} -(I-A)A_{\gamma} - A_{\gamma}A_{\gamma}\\
& = (I-A) - (A_{\gamma}-A)A_{\gamma},
\end{split}\end{equation*}
which is equivalent to
\begin{equation}\label{equ:Bm}
B_{\gamma}(I-A_{\gamma}) = I-S_{\gamma},
\end{equation}
where 
$$S_{\gamma}:= (I-A)^{-1}(A_{\gamma}-A)A_{\gamma}.$$
From assumption \eqref{equ:etaA} that $\|S_{\gamma}\|<1$, the Neumann series implies that $(I-S_{\gamma})^{-1}$ exists and is bounded by 
$$\|(I-S_{\gamma})^{-1}\|\leq \frac{1}{1-\|S_{\gamma}\|}.$$
Therefore, the equation \eqref{equ:Bm} suggests that $I-A_{\gamma}$ is an injection and then, since $A_{\gamma}$ is assumed to be compact, $(I-A_{\gamma})^{-1}$ exists. Equation \eqref{equ:Bm} also yields 
$$(I-A_{\gamma})^{-1} = (I-S_{\gamma})^{-1}B_{\gamma},$$
leading to the estimate \eqref{equ:I-Am}. The error estimate \eqref{equ:lemerror} then follows from \eqref{equ:subtracting}.
\end{proof}

\begin{remark}
This lemma is motivated by the works of Brakhage \cite{MR0129147} and Anselone and Moore \cite{MR0184448}. In their works, the assumption \eqref{equ:etaA} is ensured by assuming the sequence $\{A_{\gamma}\}$ to be \emph{collectively compact} and \emph{pointwise convergent}, essentially relying on a convergent quadrature rule \eqref{equ:quad}. The whole procedure can be also found in Kress's monograph \cite{MR1723850} on integral equations. In our work, instead of working with a convergent quadrature rule, we try to ensure the assumption \eqref{equ:etaA} by controlling the geometry of $\{\x_j\}_{j=1}^m$ and choosing an appropriate degree $n$ of hyperinterpolation. The assumption \eqref{equ:etaA} is not ensured by the convergence of quadrature rules but the convergence of hyperinterpolation. The details are provided below.
\end{remark}

The assumption \eqref{equ:etaA} is ensured if we can determine $m$, $n$, and $\eta$ such that
\begin{equation}\label{equ:aug}
c\|(A_{\gamma}-A)A_{\gamma}\|<1,
\end{equation}
where $c$ is the upper bound of $\|(I-A)^{-1}\|$. For this goal, we first investigate some properties of the sequence $\{A_{\gamma}\}$.

\begin{lemma}\label{lem:colcom}
The sequence $\{A_{\gamma}\}_{\gamma\in\Gamma}$ is collectively compact, i.e., for each bounded $U\subset C(\St)$, the image set $\{A_{\gamma}\varphi:\varphi\in U, \gamma\in\Gamma\}$ is relatively compact.
\end{lemma} 

\begin{proof}
Recall the construction \eqref{equ:derivation} of $W_j(\x)$. Replacing $K(\x,\y)\varphi(\y)$ with a continuous function $g(\y)$ with $\|g\|_{L^{\infty}}=1$ satisfying, for $j=1,2,\ldots,m$, 
\begin{equation*}
g(\x_j)=\begin{cases}
1&\text{if }W_j(\x)\geq 0,\\
-1&\text{if }W_j(\x)< 0.
\end{cases}
\end{equation*}
Then for all $\x\in\St$ and all $\gamma \in \Gamma$,
\begin{equation*}\begin{split}
\sum_{j=1}^m|W_j(\x)|&=\sum_{j=1}^mW_j(\x)g(\x_j) = \int_{\St}h(|\x-\y|)\mathcal{L}_ng(\y)d\omega(\y) \\
&\leq \left(2\pi\int_{-1}^1|h(2^{1/2}(1-t)^{1/2})|^2dt\right)^{1/2}\sqrt{1+\eta}\left(\sum_{j=1}^mw_j\right)^{1/2}\leq C
\end{split}\end{equation*}
for some constant $C$, where we apply the Cauchy--Schwarz inequality, the Funk--Hecke formula, and the stability result \eqref{equ:hyperstability} of hyperinterpolation. Thus we can estimate
\begin{equation}\label{equ:Ambounded}
\|A_{\gamma}\varphi\|_{L^{\infty}}\leq C\max_{\x,\y\in\St}|K(\x,\y)|\|\varphi\|_{L^{\infty}}.
\end{equation}

Similarly defining another $g$ for $W_j(\tilde{\x})-W_j(\x)$, we have, for all $\tilde{\x},\x\in\St$ and all $\gamma \in \Gamma$,
\begin{equation*}\begin{split}
\sum_{j=1}^m|W_j(\tilde{\x})-W_j(\x)| 
&= \sum_{j=1}^m(W_j(\tilde{\x})-W_j(\x))g(\x_j) \\
& = \sum_{\ell=0}^n\sum_{k=1}^{2\ell+1} \left(\int_{\St}\left(h(|\tilde{\x}-\y|)-h(|\x-\y|)\right) Y_{\ell,k}(\y){ d}\omega(\y)\right)\left\langle g,Y_{\ell,k}\right\rangle_m\\
& = \sum_{\ell=0}^n\sum_{k=1}^{2\ell+1} \mu_{\ell}\left(Y_{\ell,k}(\tilde{\x}) - Y_{\ell,k}({\x})\right)\left\langle g,Y_{\ell,k}\right\rangle_m\\
& \leq \left(\sum_{\ell=0}^n\sum_{k=1}^{2\ell+1} \mu_{\ell}^2\left(Y_{\ell,k}(\tilde{\x}) - Y_{\ell,k}({\x})\right)^2\right)^{1/2}\|\mathcal{L}_ng\|_{L^{2}},
\end{split}\end{equation*} 
where 
$$\mu_{\ell} = 2\pi\int_{-1}^1h(2^{1/2}(1-t)^{1/2})P_{\ell}(t)dt.$$
Since 
$$\left(\sum_{\ell=0}^{\infty}\sum_{k=1}^{2\ell+1} \mu_{\ell}^2\left(Y_{\ell,k}(\tilde{\x}) - Y_{\ell,k}({\x})\right)^2\right)^{1/2}
=  \|h(|\tilde{\x}-\y|)-h(|\x-\y|)\|_{L^2_{\y}}<\infty$$
and each $Y_{\ell,k}$ is uniformly continuous, we have 
\begin{equation}\label{equ:Amconv}
\lim_{\tilde{\x}\rightarrow \x}\sup_{\gamma\in\Gamma}\sum_{j=1}^m|W_j(\tilde{\x})-W_j(\x)| =0
\end{equation}
uniformly for all $\x\in \St$. Then, from
\begin{equation*}
(A_{\gamma}\varphi)(\tilde{\x}) - (A_{\gamma}\varphi)(\x) = \sum_{j=1}^mW_j(\tilde{\x})(K(\tilde{\x},\x_j) - K({\x},\x_j))\varphi(\x_j) +\sum_{j=1}^m(W_j(\tilde{\x})-W_j(\x))K(\x,\x_j)\varphi(\x_j),
\end{equation*}
we can estimate
\begin{equation*}\begin{split}
|(A_{\gamma}\varphi)(\tilde{\x}) - &(A_{\gamma}\varphi)(\x)|\leq\\
& C \max_{\y\in\St}|K(\tilde{\x},\y)- K(\x,\y)|\|\varphi\|_{L^{\infty}}+\max_{\x,\y\in\St}|K(\x,\y)|\sup_{\gamma\in\Gamma}\sum_{j=1}^m|W_j(\tilde{\x})-W_j(\x)|\|\varphi\|_{L^{\infty}}
\end{split}\end{equation*}
for all $\tilde{\x},\x\in\St$. Now let $U\in C(\St)$ be bounded. Then from \eqref{equ:Ambounded} and \eqref{equ:Amconv}, we see that $\{A_{\gamma}\varphi:\varphi\in U, \gamma\in\Gamma\}$ is bounded and equicontinuous because $K$ is uniformly continuous on $\St\times \St$ and $\sup_{\gamma\in\Gamma}\sum_{j=1}^m|W_j(\tilde{\x})-W_j(\x)|$ converges uniformly to 0 as $\tilde{\x}$ tends to $\x$ for all $\x\in\St$ and all $\varphi\in U$. Then by the Arzel\`{a}--Ascoli theorem, $\{A_{\gamma}\varphi:\varphi\in U, \gamma\in\Gamma\}$ is relatively compact and hence the sequence $\{A_{\gamma}\}_{\gamma\in\Gamma}$ is collectively compact.
\end{proof}

In the following lemma, we investigate the limiting case where $n\rightarrow\infty$ and $\eta\rightarrow0$ simultaneously. For $\gamma\in\Gamma$, due to the Markinciewicz--Zygmund property \eqref{equ:MZ}, this case renders $m\rightarrow\infty$.

\begin{lemma}\label{lem:conv}
The sequence $\{A_{\gamma}\}_{\gamma\in\Gamma}$ is pointwise convergence as $n\rightarrow\infty$ and $\eta\rightarrow0$ simultaneously, i.e., $A_{\gamma}\varphi\rightarrow A\varphi$ for all $\varphi\in C(\St)$ as $n\rightarrow\infty$ and $\eta\rightarrow0$ simultaneously.
\end{lemma} 

\begin{proof}
For fixed $\varphi\in C(\St)$, we show that the sequence $\{A_{\gamma}\varphi\}_{\gamma\in \Gamma}$ is pointwise convergent as $n\rightarrow\infty$ and $\eta\rightarrow0$, i.e., $(A_{\gamma}\varphi)(\x)\rightarrow (A\varphi)(\x)$ in the limiting case, for all $\x\in\St$. Indeed, for each $\x\in\St$, the $L^2$ error analysis of hyperinterpolation indicates that 
\begin{equation*}\begin{split}
|(A_{\gamma}\varphi)(\x) - (A\varphi)(\x)|
&= \left|\int_{\St}h(|\x-\y|)\mathcal{L}_n(K(\x,\y)\varphi(\y))d\omega(\y) -\int_{\St}h(|\x-\y|)K(\x,\y)\varphi(\y)d\omega(\y)  \right|\\
& \leq \|h(\x,\y)\|_{L^2_{\y}} \|\mathcal{L}_n(K(\x,\y)\varphi(\y)) -  K(\x,\y)\varphi(\y)\|_{L^2_{\y}}\\
& \leq C \|h(\x,\y)\|_{L^2_{\y}} \left(E_n(K(\x,\y)\varphi) + \sqrt{\eta^2+4\eta}\|\chi^*\|_{L^{2}_{\y}}\right)
\end{split}\end{equation*}
for some constant $C>0$, where $\chi^*(\y)$ stands for the best approximation polynomial of $K(\x,\y)\varphi(\y)$ in $\Pn$. 

As shown in the above proof of Lemma \ref{lem:colcom}, the sequence $\{A_{\gamma}\varphi\}_{\gamma\in \Gamma}$ is equicontinuous. Hence the convergence of this sequence can be upgraded to uniform convergence, i.e., for fixed $\varphi\in C(\St)$, $\|A_{\gamma}\varphi - A\varphi\|_{L^{\infty}}\rightarrow 0$ as $n\rightarrow\infty$ and $\eta\rightarrow0$ simultaneously. Therefore, we have the pointwise convergence of the sequence $\{A_{\gamma}\}_{\gamma\in \Gamma}$ of numerical integration operators, that is, $A_{\gamma}\varphi\rightarrow A\varphi$, for all $\varphi\in C(\St)$, as $n\rightarrow\infty$ and $\eta\rightarrow0$.
\end{proof}

Restricting to a compact set $U\subset C(\St)$, the next lemma shows that the pointwise convergence of $\{A_{\gamma}\}_{\gamma\in \Gamma}$ can be further upgraded to uniform convergence.

\begin{lemma}\label{lem:unifboundedcoro}
Let $U$ be a compact subset of $C(\St)$. Then $\{A_{\gamma}\}_{\gamma\in \Gamma}$ has uniform convergence on $U$ to $A$, i.e., 
$$\sup_{\varphi\in U}\|A_{\gamma}\varphi - A\varphi\|_{L^{\infty}}\rightarrow 0$$
as $n\rightarrow\infty$ and $\eta\rightarrow0$ simultaneously.
\end{lemma}

\begin{proof}
From the bound \eqref{equ:Ambounded} on each $A_{\gamma}$, the uniform boundedness principle results in a uniform bound on the operator norms of all $A_{\gamma}$, i.e., there exists some constant $C>0$ such that $\|A_{\gamma}\|\leq C$ for all $\gamma\in\Gamma$. For any $\varepsilon>0$, consider open balls $B(\varphi,r)=\{\psi\in C(\St):\|\varphi-\psi\|_{L^{\infty}}<r\}$ with center $\varphi\in C(\St)$ and radius $r = \varepsilon/(3C)$. Due to the compactness of $U$, there exists a finite covering such that
$$U\subset \bigcup_{k=1}^{M} B(\varphi_k,r).$$
Pointwise convergence of $\{A_{\gamma}\}$, as shown in Lemma \ref{lem:conv}, suggests that there exists $(N(\varepsilon),H(\varepsilon))$ such that 
$$\|A_{\gamma}\varphi_k - A\varphi_k\|_{L^{\infty}}<\frac{\varepsilon}{3}$$
for all $n\geq N(\varepsilon)$ and $\eta\geq H(\varepsilon)$, and for all $k=1,\ldots,M$. Now let $\varphi\in U$ be arbitrary. Then $\varphi$ is contained in some ball $B(\varphi_k,r)$ with center $\varphi_k$. Hence for all $n\geq N(\varepsilon)$ and $\eta\geq H(\varepsilon)$, we have 
\begin{equation*}\begin{split}
\|A_{\gamma}\varphi-A\varphi\|_{L^{\infty}} 
&\leq \|A_{\gamma}\varphi-A_{\gamma}\varphi_k\|_{L^{\infty}} + \|A_{\gamma}\varphi_k-A\varphi_k\|_{L^{\infty}}  +\|A\varphi_k-A\varphi\|_{L^{\infty}} \\
& \leq \|A_{\gamma}\|\|\varphi-\varphi_k\|_{L^{\infty}} + \frac{\varepsilon}{3} + \|A\|\|\varphi-\varphi_k\|_{L^{\infty}}\\
&\leq 2Cr + \frac{\varepsilon}{3} = \varepsilon.
\end{split}\end{equation*}
Therefore, the convergence of $\{A_{\gamma}\}_{\gamma\in \Gamma}$ is uniform on $U$. 
\end{proof}

\begin{lemma}\label{lem:auglem}
$\|(A_{\gamma}-A)A_{\gamma}\|\rightarrow 0$ as $n\rightarrow\infty$ and $\eta\rightarrow0$ simultaneously.
\end{lemma} 

\begin{proof}
Consider the set $U:=\{A_{\gamma}\varphi:\|\varphi\|_{L^{\infty}}\leq 1, \gamma\in\Gamma\}$. By Lemma \ref{lem:colcom}, the set $U$ is relatively compact. By Lemma \ref{lem:unifboundedcoro}, the convergence $A_{\gamma}\psi\rightarrow A\psi$ as $n\rightarrow\infty$ and $\eta\rightarrow0$ is uniform for all $\psi \in U$. Hence, for any $\varepsilon>0$, there exists $(N(\varepsilon),H(\varepsilon))$ such that the sequence of functions satisfies
$$\|(A_{\gamma}-A)A\varphi\|_{L^{\infty}}<\varepsilon$$
for all $n\geq N(\varepsilon)$ and $\eta\geq H(\varepsilon)$, all $\varphi\in C(\St)$ with $\|\varphi\|_{L^{\infty}}\leq 1$, and all $\gamma\in\Gamma$. Therefore, the operator norms satisfy
$$\|(A_{\gamma}-A)A\|\leq\varepsilon$$
for all $n\geq N(\varepsilon)$ and $\eta\geq H(\varepsilon)$ and all $\gamma\in\Gamma$. 
\end{proof}

We are now prepared well for the proof of our main theorem.\\

\noindent{\emph{Proof of Theorem \ref{thm:main}.} This theorem follows from Lemma \ref{lem:main}, which guarantees the existence and boundedness of $(I-A_{\gamma})^{-1}$. The assumption \eqref{equ:etaA} of Lemma \ref{lem:main} is ensured by \eqref{equ:aug}, which is a direct consequence of Lemma \ref{lem:auglem}. Therefore, we can conclude that $(I-A_{\gamma})^{-1}$ exists and is bounded, which implies the bound \eqref{equ:stability} of $\varphi_{\gamma}$. Furthermore, combining \eqref{equ:stability} with \eqref{equ:lemerror}, we obtain the error estimate \eqref{equ:error}.\hfill $\square$

\section{Implementation and examples of kernels}\label{sec:implementation}

The explicit construction of $W_j(\x)$ relies on the explicit evaluation of the modified moments \eqref{equ:mm}. These moments are explicitly available for some specific kernels. In this section, we list some singular kernels and elaborate the evaluation of the corresponding moments.

\subsection{Non-singular case}

If $h\equiv 1$, since
$$\sum_{\ell=0}^n\sum_{k=1}^{2\ell+1}\left(\int_{\St}h(|\x-\y|) Y_{\ell,k}(\y){ d}\omega(\y)\right)Y_{\ell,k}(\x_j)  = 1,$$
we have
$$W_j(\x)=w_j\left(\sum_{\ell=0}^n\sum_{k=1}^{2\ell+1} \left(\int_{\St}h(|\x-\y|) Y_{\ell,k}(\y){ d}\omega(\y)\right)Y_{\ell,k}(\x_j)\right) \equiv w_j.$$

\subsection{Singular cases}

We explore three common singular kernels in the context of integral equations.

\noindent \textbf{Case 1:} $h(|\x-\y|) = |\x-\y|^{\nu}$ with $-1<\nu<0$. The assumption $\nu>-1$ is sufficient for the $L^2$ assumption of $h$. For the case of $\nu\geq0$, it corresponds to a continuous kernel. For the modified moments \eqref{equ:mm} with this specific kernel, we use the Rodrigues representation formula
$$
P_{\ell}(t)=\frac{(-1)^{\ell}}{2^{\ell}{\ell}!}\left(\frac{d}{d t}\right)^{\ell}\left(1-t^2\right)^{\ell}
$$
for Legendre polynomial $P_{\ell}(t)$ of degree $\ell$. Then we can express the constant $\mu_{\ell}$ in the form of
$$\mu_{\ell}=2\pi\frac{(-1)^{\ell}}{2^{\ell}{\ell}!} \int_{-1}^1 2^{\nu / 2}(1-t)^{\nu / 2}\left(\frac{d}{d t}\right)^{\ell}\left(1-t^2\right)^{\ell} d t,$$
which can be further simplified through computing the integral $$I(\nu)=\int_{-1}^1(1-t)^{\nu / 2}\left(\frac{d}{d t}\right)^{\ell}\left(1-t^2\right)^{\ell} d t.$$ Under the condition $\nu>-1$, we can perform integration by parts repeatedly on $I(\nu)$ and all terms containing $(1-t^2)$ vanish at $t= \pm 1$. After integrating by parts $\ell$ times, we have $$I(\nu)=\frac{\nu}{2}\left(\frac{\nu}{2}-1\right) \cdots\left(\frac{\nu}{2}-(\ell-1)\right) J(\nu)=(-1)^{\ell}\left(-\frac{\nu}{2}\right)_{\ell} J(\nu),$$
where $(x)_n:=\frac{\Gamma(x+n)}{\Gamma(x)}$ with $x>0$ stands for Pochhammer’s symbol and
$$J(\nu):=\int_{-1}^1(1-t)^{\nu / 2-\ell}\left(1-t^2\right)^{\ell} d t=\int_{-1}^1(1-t)^{\nu / 2}(1+t)^{\ell} d t.$$
The term $J(\nu)$ can by further evaluated, via the change of variables $t=2 s-1$, as
$$J(\nu)  =2^{\ell+\frac{\nu}{2}+1} \int_0^1(1-s)^{\frac{\nu+2}{2}-1} s^{\ell} d s  =2^{\ell+\frac{\nu}{2}+1} \frac{\Gamma\left(\frac{\nu+2}{2}\right) \Gamma\left(\ell+1\right)}{\Gamma\left(\ell+\frac{\nu}{2}+2\right)} .
$$
Therefore the modified moment \eqref{equ:mm} of $h(|\x-\y|) = |\x-\y|^{\nu}$ with $-1<\nu<0$ can be evaluated analytically as
$$\mu_{\ell}= 2^{\nu+2}\pi \left(-\frac{\nu}{2}\right)_{\ell}  \frac{\Gamma\left(\frac{\nu+2}{2}\right)}{\Gamma\left(\ell+\frac{\nu}{2}+2\right)}.$$
More details on this derivation can be found in \cite[Section 3.7.1]{MR2934227}.

\noindent \textbf{Case 2:} $h(|\x-\y|) = \log|\x-\y|$. For the modified moments \eqref{equ:mm} with this specific kernel, note that $\log(2^{1/2}(1-t)^{1/2})$ satisfies the requirement \eqref{equ:h}. Thus we have 
$$\mu_{\ell}=2\pi \int_{-1}^1 \log(2^{1/2}(1-t)^{1/2})P_{\ell}(t) d t =\pi \int_{-1}^1 \log (2(1-t)) P_{\ell}(t)d t.$$
For $\ell = 0$, we have $$\mu_0 = \pi \int_{-1}^1\log(2(1-t))dt= \pi\left(2\log{2} + \int_{-1}^1\log(1-t)dt\right)= \pi(4\log{2}-2).$$
For $\ell\geq 1$, we utilize the orthogonality of Legendre polynomials and the Maclaurin series of $\log(1-t)$ for
$$\mu_\ell = -\pi \int_{-1}^1\sum_{k = 1}^{\infty}\frac{t^{k}}{k}P_{\ell}(t)dt =-\pi \sum_{k = 1}^{\ell} \frac{1}{k}\int_{-1}^1t^kP_{\ell}(t)dt,\quad \ell = 1,2,\ldots.$$
From our numerical experience, replacing $\log(1-t)$ with its Maclaurin series should be necessary. To compute integrals $\int_{-1}^1t^kP_{\ell}(t)dt$, we work with Chebfun \cite{driscoll2014chebfun}. Here is a Matlab code segment that evaluates $\mu_{\ell}$ for $\ell=0,1,\ldots,n$.

\begin{lstlisting}[language=Matlab]
x = chebfun('x');
mu(1) = pi*(4*log(2)-2);
for ell = 1:n
  mu_temp = 0;
  for k = 1:ell
    mu_temp = mu_temp + sum(x.^k.*legpoly(ell))/k;
  end
  mu(ell+1) = -pi * mu_temp;
end
\end{lstlisting}

\noindent \textbf{Case 3:} $h(|\x-\y|) = |\x-\y|^{\nu_1}|\x+\y|^{\nu_2}$ with $-1\leq \nu_1,\nu_2<0$. Similar to Case 2, we compute the modified moments \eqref{equ:mm} with this specific kernel in terms of 
$$\mu_{\ell} = 2^{\left(\nu_1+\nu_2\right) / 2}(2\pi) \int_{-1}^1(1-t)^{\nu_1 / 2}(1+t)^{\nu_2 / 2}P_{\ell}(t) dt.$$
Specifying $\nu_1$ and $\nu_2$ as \texttt{nu1} and \texttt{nu2}, a Matlab code segment for $\mu_{\ell}$ with $\ell=0,1,\ldots,n$ is provided below.
\begin{lstlisting}[language=Matlab]
x = chebfun('x');
const = 2^((nu1+nu2)/2)*2*pi;
for ell = 0:n
  mu(ell+1) = const*sum((1-x).^(nu1/2).*(1+x).^(nu2/2).*legpoly(ell));
end
\end{lstlisting}

\section{Numerical experiments}\label{sec:numerical}
Here we give numerical results for our two-stage scheme \eqref{equ:scheme1} and \eqref{equ:scheme2} in solving the integral equation \eqref{equ:equation}, defined using both non-singular kernels and the three types of singular kernels mentioned above. We investigate four kinds of point distributions on sphere:
\begin{itemize}
       \item Spherical $t$-designs \cite{delsarte1991geometriae}: a set of points $\{x_j\}_{j=1}^m\subset \St$ with the characterizing property that an equal-weight quadrature rule in these points exactly integrates all polynomials of degree at most $t$, that is,
\begin{equation}\label{equ:stdexactness}
\frac{4\pi}{m}\sum_{j=1}^m\chi(\x_j)=\int_{\St}\chi(\x)\text{d}\omega(\x)\quad\forall \chi\in\mathbb{P}_t.
\end{equation}
In our experiment, we employ well conditioned spherical $t$-designs \cite{MR2763659} to realize the approximation;
       \item Minimal energy points \cite{MR2065291,MR1845243}: a set of points $\{x_j\}_{j=1}^m\subset \St$ that minimizes the Coulomb energy $\sum_{i,j=1}^m1/\|\x_i-\x_j\|_2$;
       \item Fekete points \cite{MR2065291}: a set of points that maximizes the determinant for polynomial interpolation.
       % $$\det\left(Y_{\ell,k}(\x_j)\right)_{\ell=0,\ldots,n,k=1,\ldots,2\ell+1,j=1,2,\ldots,(n+1)^2}.$$
       \item Equal area points \cite{MR1306011} based on an algorithm given in \cite{MR2582801};
\end{itemize}
These four kinds of points are utilized in the first-stage \eqref{equ:scheme1} of our scheme. For the second-stage \eqref{equ:scheme2} evaluating values of the numerical solution, we adopt 5,000 uniformly distributed points on $\St$.

\noindent \textbf{Experiment 1:} $h(\x,\y)\equiv 1$. We begin by testing our scheme using the case where the solution is constant, specifically $\varphi\equiv 1$. Though $h$ is non-singular, we let the continuous kernel be defined as $K(\x,\y)=\sin(10|\x-\y|)$, exhibiting oscillatory behavior. Analytically, based on the Funk--Hecke formula, we determine that the inhomogeneous term $f$ should also be constant:
$$f(\x,\y)= 1- 2\pi\int_{-1}^1\sin\left(10\sqrt{2(1-t)}\right)dt\approx\texttt{1.455449001125579}.$$
We set the degree of hyperinterpolation to $n=20$ and choose the number of quadrature points as $m=(2n+1)^2$. Figure \ref{fig:example1} illustrates the numerical solutions of the integral equation \eqref{equ:equation}. We consider four types of quadrature points, comparing their errors with respect to the exact solution $\varphi= 1$. Notably, the numerical solutions closely approximate the constant value of 1, as indicated by the color bar ranges, with the scheme utilizing spherical $t$-designs performing the best.
\begin{figure}[htbp] 
  \centering
  % \hspace{-0.1\textwidth}
  % include first image
  \includegraphics[width=\textwidth]{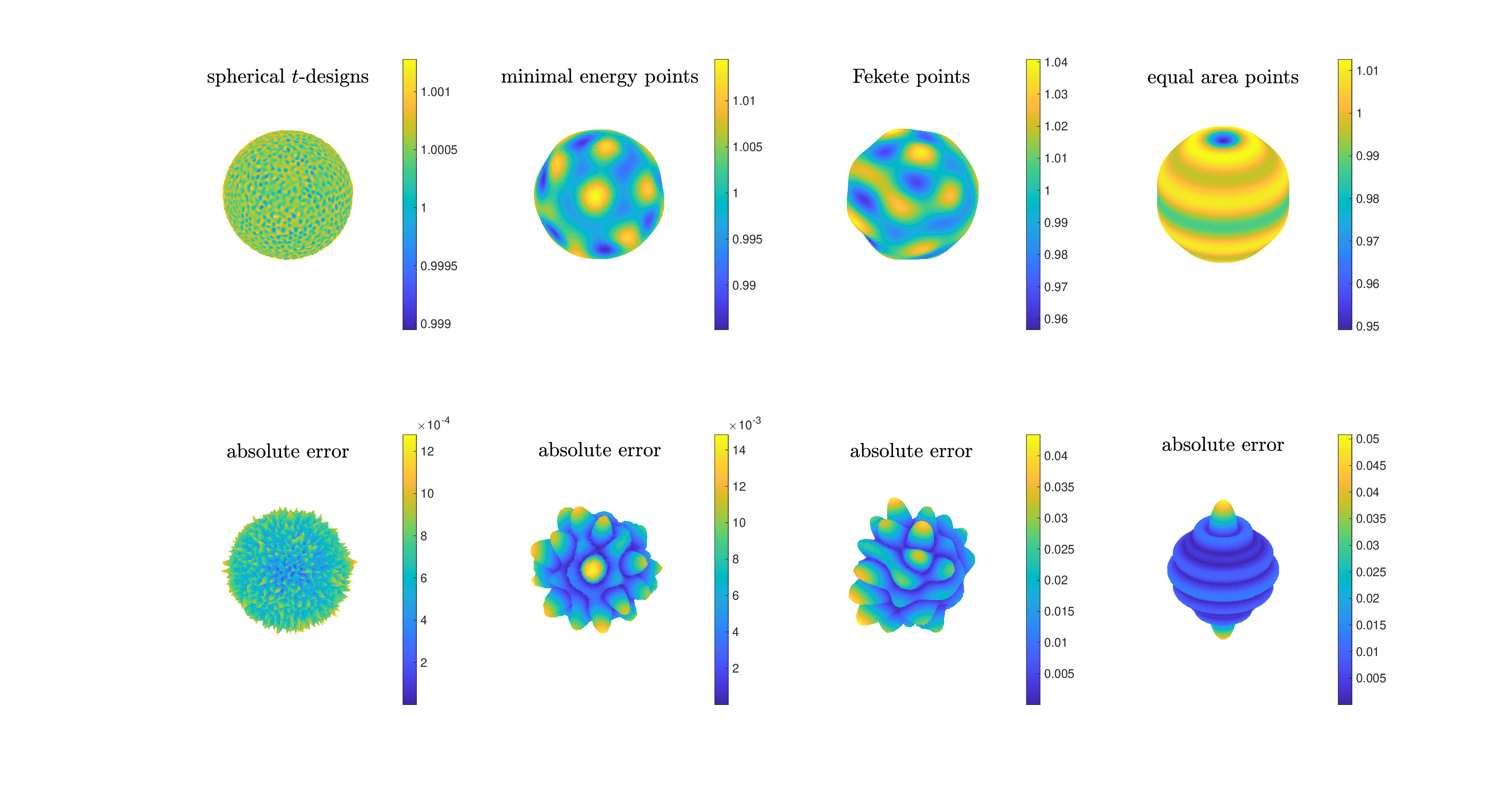}\vspace{-1cm}
  \caption{Non-singular $h=1$ and oscillatory $K(\x,\y)=\sin(10|\x-\y|)$: Numerical solutions ($n=20$ and $m = (2n+1)^2$) of the equation \eqref{equ:equation} with exact solution being constant 1 and their absolute errors.}\label{fig:example1}
\end{figure}

To further explore their comparison and validate our theoretical error analysis, we consider uniform errors ($L^{\infty}$ errors) of numerical solutions for varying $n$ and $m$. The first plot in Figure \ref{fig:example11} investigates the errors of numerical solutions of the equation considered in Figure \ref{fig:example1} with fixed $n = 20$ and varying $m$ in the sense that $m = (t + 1)^2$ for $t$ ranging from 15 to 35. The second plot demonstrates errors with $n$ ranging from 15 to 35 and $m$ varying corresponding to $n$ in terms of $m = (\lfloor 1.2n \rfloor + 1)^2$. Figure \ref{fig:example11} not only enriches the comparison results among quadrature points but also validates our theoretical analysis that the $L^{\infty}$ errors decrease as $n$ grows and the constant $\eta$ in the Marcinkiewicz--Zygmund property \eqref{equ:MZ} decreases. The decrease in $\eta$ is a consequence of an increasing number $m$ of points for a well-distributed point set. To ensure that the quadrature rule \eqref{equ:quad} using these points satisfies the Marcinkiewicz–Zygmund property \eqref{equ:MZ} for polynomials of degree at most $n$, $m$ has to increase correspondingly to $n$.

\begin{figure}[htbp] 
  \centering
  \includegraphics[width=\textwidth]{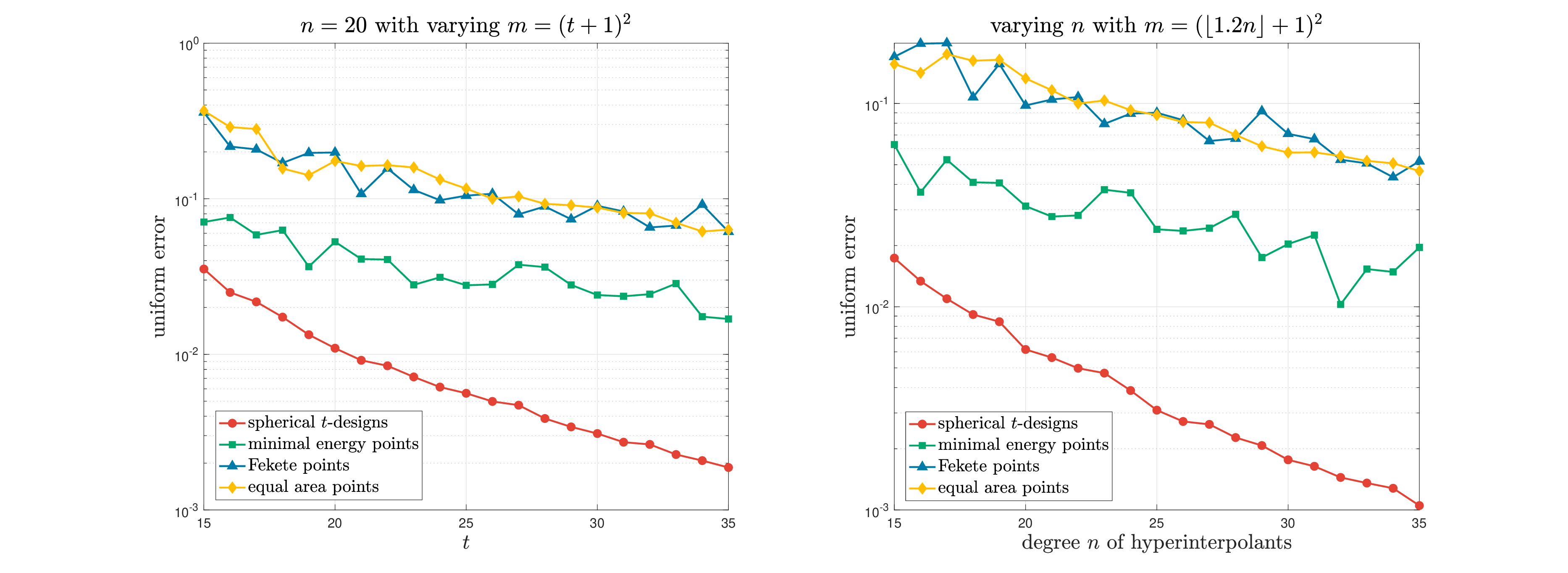}
  \caption{Non-singular $h=1$ and oscillatory $K(\x,\y)=\sin(10|\x-\y|)$: Uniform errors with different $n$ and $m$.}\label{fig:example11}
\end{figure}

\noindent \textbf{Experiment 2:} $h(\x,\y)= |\x-\y|^{\nu}$ \textbf{with} $-1<\nu<0$. We continue by considering a singular kernel $h(\x, \y) = |\x - \y|^{\nu}$ with $\nu = -0.5$. Similar to Experiment 1, we test our scheme using the case where the solution is constant, $\varphi \equiv 1$. We consider an oscillatory kernel again, $K(\x, \y) = \cos(10|\x - \y|)$. The inhomogeneous term $f$ now should be
$$f(\x, \y) \equiv 1 - 2\pi \int_{-1}^1 \left(\sqrt{2(1-t)}\right)^{-0.5} \cos\left(10\sqrt{2(1-t)}\right) dt \approx \texttt{0.303738699125466}.$$

We set the degree of hyperinterpolation to $n = 20$ and choose the number of quadrature points as $m = (2n + 1)^2$. Figure \ref{fig:example2} illustrates the numerical solutions of the integral equation \eqref{equ:equation} with their absolute errors. These solutions are still close to the exact solution $\varphi \equiv 1$, and again, the scheme utilizing spherical $t$-designs performs the best. 

\begin{figure}[htbp] 
  \centering
  % \hspace{-0.1\textwidth}
  % include first image
  \includegraphics[width=\textwidth]{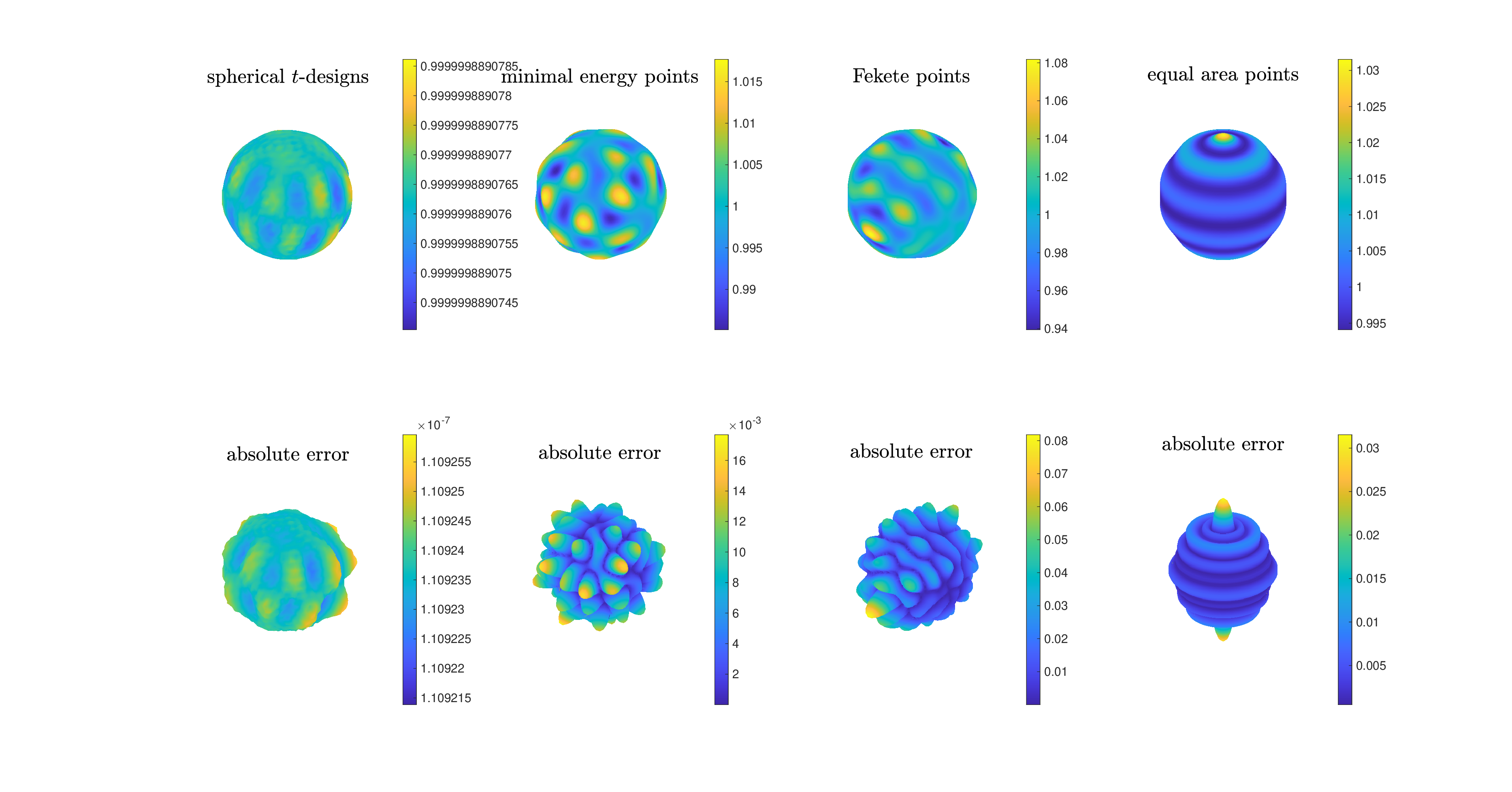}\vspace{-1cm}
  \caption{Singular $h(\x,\y)= |\x-\y|^{-0.5}$ and oscillatory $K(\x,\y)=\cos(10|\x-\y|)$: Numerical solutions ($n=20$ and $m = (2n+1)^2$) of the equation \eqref{equ:equation} with exact solution being constant 1 and their absolute errors.}\label{fig:example2}
\end{figure}

Similar to Figure \ref{fig:example11}, we validate our theoretical analysis for varying $n$ and $m$ in the case of $h(\x,\y)= |\x-\y|^{-0.5}$. This validation is demonstrated in Figure \ref{fig:example21}.  

\begin{figure}[htbp] 
  \centering
  % \hspace{-0.1\textwidth}
  % include first image
  \includegraphics[width=\textwidth]{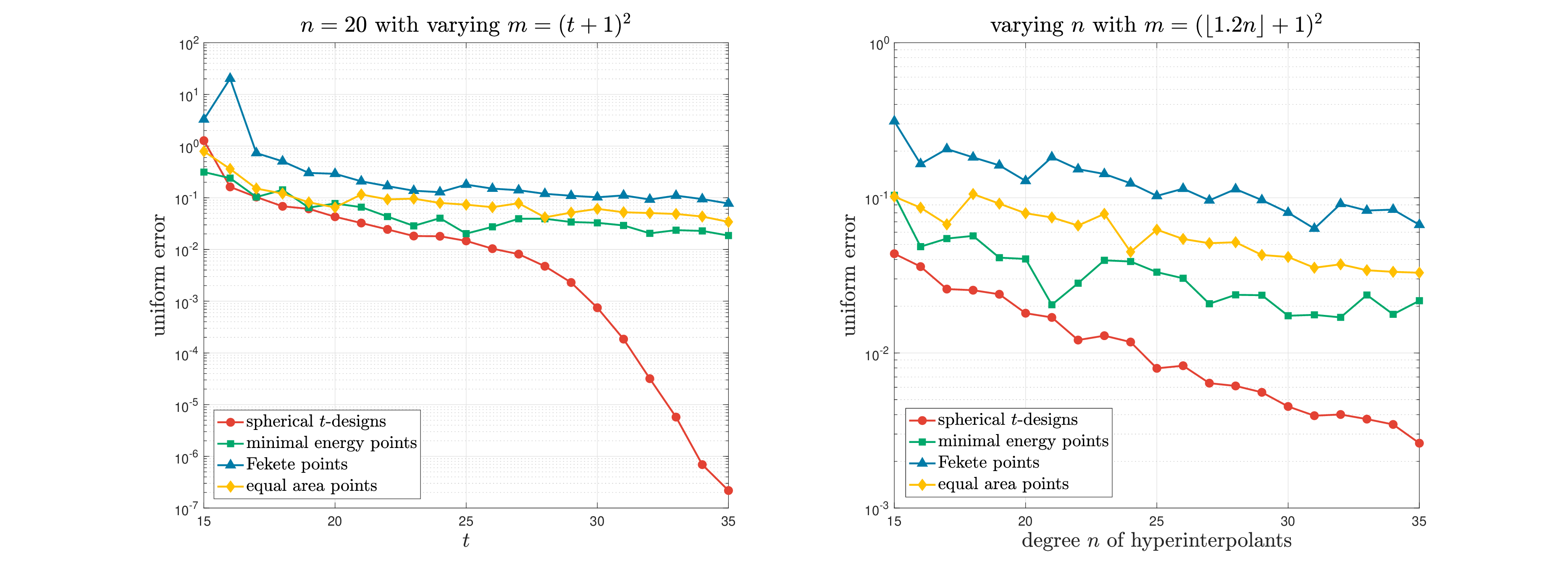}
  \caption{Singular $h(\x,\y)= |\x-\y|^{-0.5}$ and oscillatory $K(\x,\y)=\cos(10|\x-\y|)$: Uniform errors with different $n$ and $m$.}\label{fig:example21}
\end{figure}

\noindent \textbf{Experiment 3:} $h(\x,\y)= \log|\x-\y|$. We then consider another singular kernel, the log kernel (with natural base $e$). Similar to previous experiments, we test our scheme using the case where the solution is $\varphi \equiv 1$. In this experiment, we let $K = 1$. The reason for considering the non-oscillatory kernel $K$ is to test the accuracy of our scheme. Recall our error bound \eqref{equ:error}: if $K = 1$ and $\varphi = 1$, then $E_n(K(\x_0, \cdot)\varphi) = 0$ for $n\geq 1$; and if the quadrature rule is exact for polynomials of degree $2n$, then $\sqrt{\eta^2 + 4\eta}\|\chi^*\|_{L^{2}} = 0$. Thus, in the above situation, we should observe the error to be near machine epsilon. If $K$ contains oscillatory behavior, then similar results to the previous two experiments have been observed. In this case, note that the inhomogeneous term 
$$f(\x, \y) \equiv 1 - 2\pi \int_{-1}^1 \log\left(\sqrt{2(1-t)}\right) dt = 1-\pi(4\log2-2).$$

We set the degree of hyperinterpolation to $n = 5$ and choose the number of quadrature points as $m = (2n + 1)^2$. Figure \ref{fig:example3} illustrates the numerical solutions of the integral equation \eqref{equ:equation} with their absolute errors. The solution using spherical $t$-design is computed with error near machine epsilon. We further explore this interesting fact with varying $n$ and $m$. 

\begin{figure}[htbp] 
  \centering
  % \hspace{-0.1\textwidth}
  % include first image
  \includegraphics[width=\textwidth]{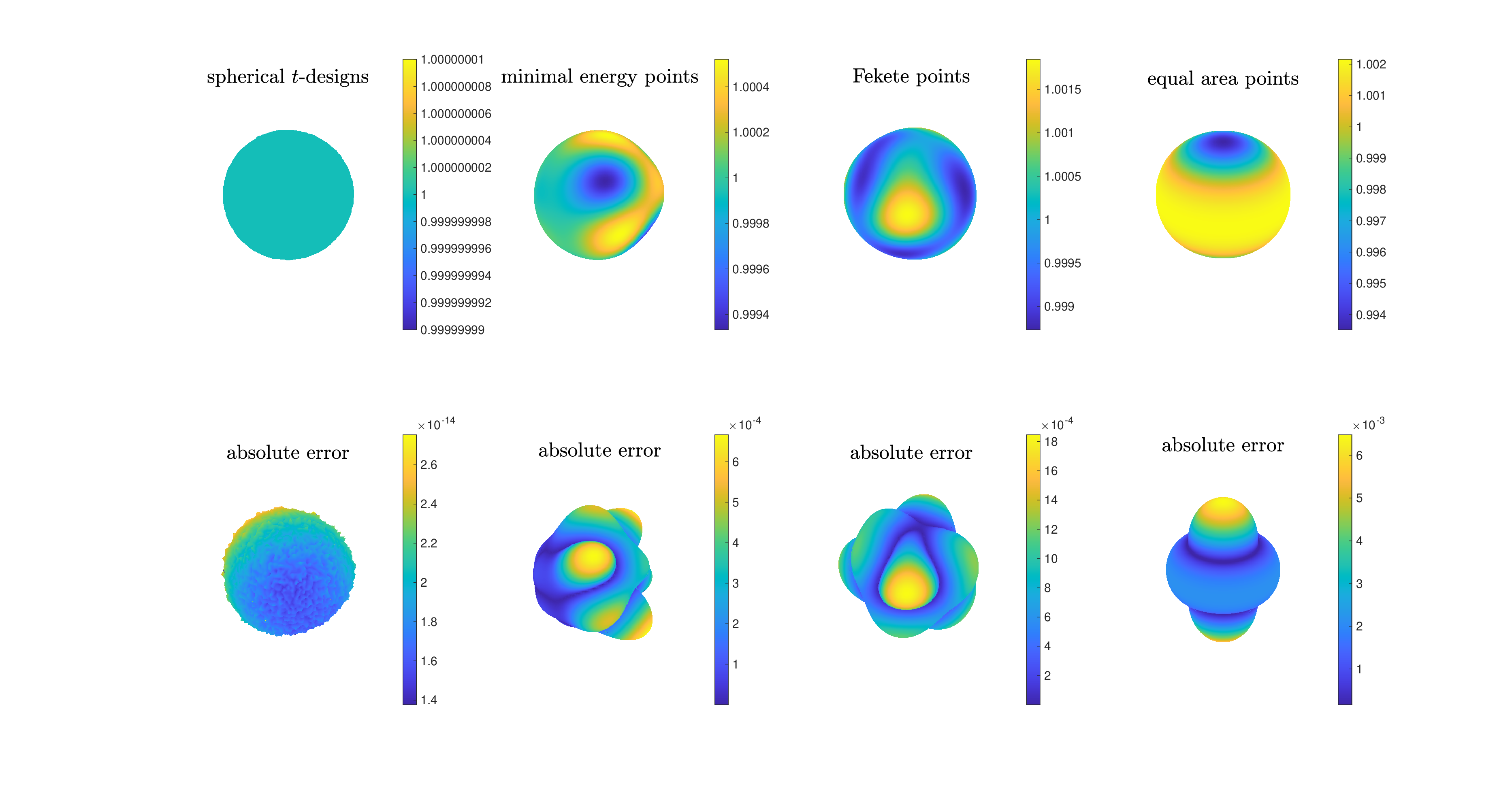}\vspace{-1cm}
  \caption{Singular $h(\x,\y)= \log|\x-\y|$ and non-oscillatory $K=1$: Numerical solutions ($n=5$ and $m = (2n+1)^2$) of the equation \eqref{equ:equation} with exact solution being constant 1 and their absolute errors.}\label{fig:example3}
\end{figure}

The uniform errors of computed numerical solutions with varying $n$ and $m$ for the log kernel are depicted in Figure \ref{fig:example31}. Our error analysis is validated in Figure \ref{fig:example31} again. Notably, the error plot related to spherical $t$-designs rapidly approaches the machine epsilon as explained before. Subsequently, this error curve starts to grow slowly due to the accumulation of round-off errors as the size of the linear system in our scheme increases. 

\begin{figure}[htbp] 
  \centering
  % \hspace{-0.1\textwidth}
  % include first image
  \includegraphics[width=\textwidth]{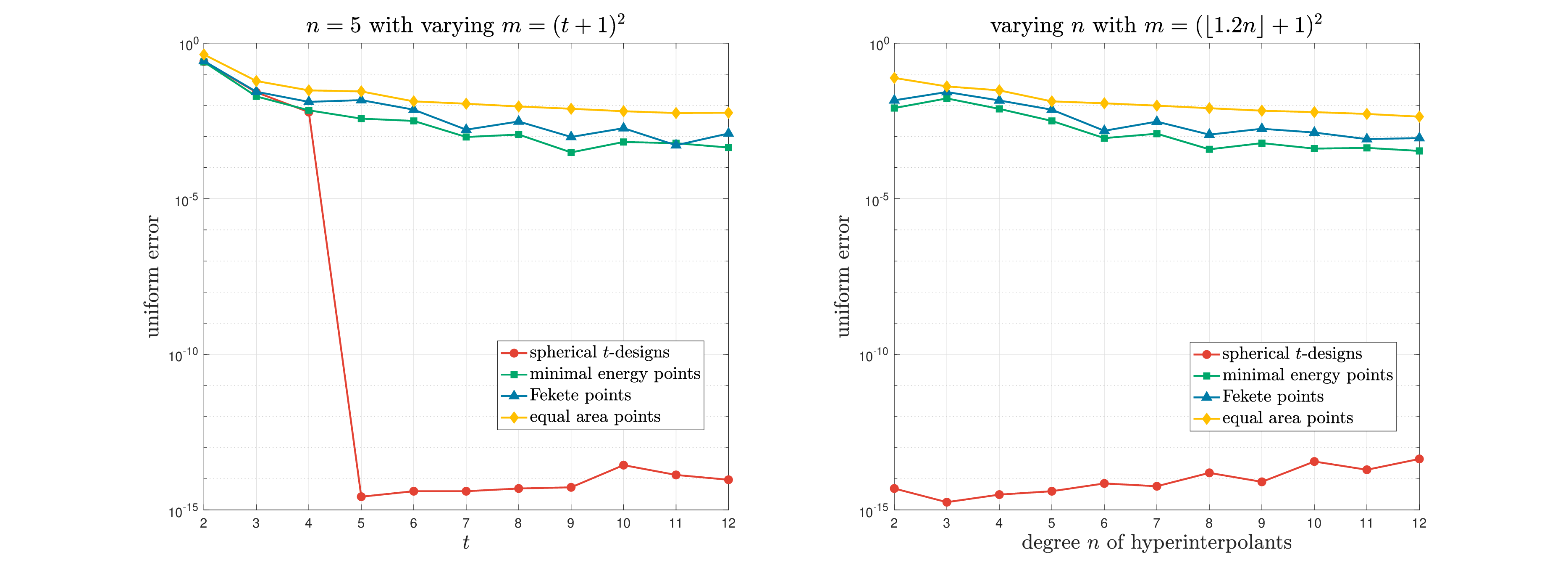}
  \caption{Singular $h(\x,\y)= \log|\x-\y|$ and non-oscillatory $K=1$: Uniform errors with different $n$ and $m$.}\label{fig:example31}
\end{figure}

\noindent \textbf{Experiment 4:} $h(\x,\y)= |\x-\y|^{\nu_1}|\x+\y|^{\nu_2}$. We then conduct an experiment under the same conditions as Experiment 1, but with the singular kernel modified to $|\x-\y|^{-0.5}|\x+\y|^{-0.5}$. For \( n = 20 \), the numerical solutions and their errors are displayed in Figure \ref{fig:example4}. Errors for varying $n$ and $m$ are shown in Figure \ref{fig:example41}. The conclusions drawn are consistent with those of the previous experiments.
\begin{figure}[htbp] 
  \centering
  % \hspace{-0.1\textwidth}
  % include first image
  \includegraphics[width=\textwidth]{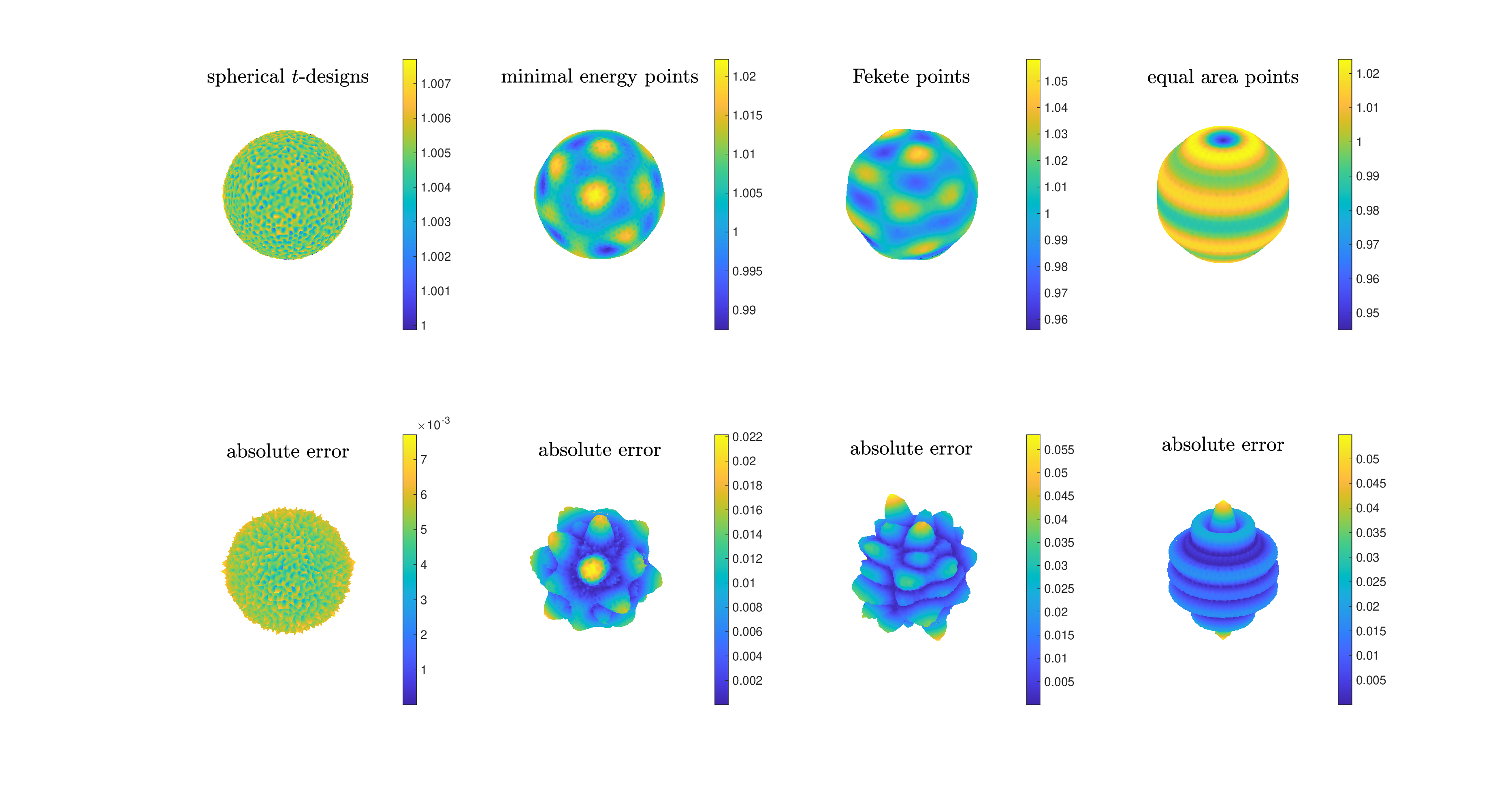}\vspace{-1cm}
  \caption{Singular $h(\x,\y)= |\x-\y|^{-0.5}|\x+\y|^{-0.5}$ and oscillatory $K(\x,\y)=\sin(10|\x-\y|)$: Numerical solutions ($n=20$ and $m = (2n+1)^2$) of the equation \eqref{equ:equation} with exact solution being constant 1 and their absolute errors.}\label{fig:example4}
\end{figure}

\begin{figure}[htbp] 
  \centering
  % \hspace{-0.1\textwidth}
  % include first image
  \includegraphics[width=\textwidth]{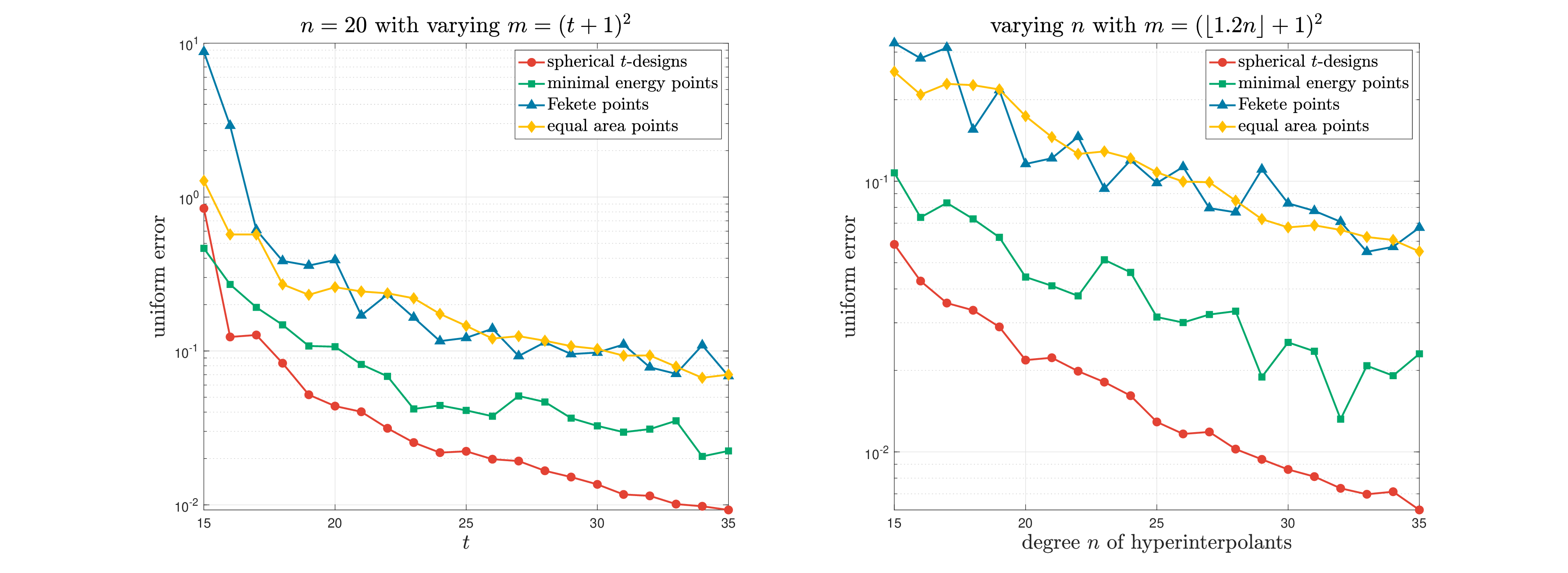}
  \caption{Singular $h(\x,\y)= |\x-\y|^{-0.5}|\x+\y|^{-0.5}$ and oscillatory $K(\x,\y)=\sin(10|\x-\y|)$: Uniform errors with different $n$ and $m$.}\label{fig:example41}
\end{figure}

\section{Discussion}\label{sec:discussion}

In this paper, we propose a two-stage numerical scheme \eqref{equ:scheme1} and \eqref{equ:scheme2} for solving the integral equation \eqref{equ:equation} of the second kind on the sphere $\St$ and provide numerical analysis. This scheme combines hyperinterpolation and product integration rules to provide an explicit evaluation of the weights \eqref{equ:Wj} $W_j(\x)$. We also present numerical experiments to show the accuracy of our scheme and validate our analysis.

In the construction of $W_j(\x)$, the distribution of quadrature points $\mathcal{X}_m$ and the evaluation of modified moments \eqref{equ:mm} are essential. Traditionally in numerical analysis, quadrature points are chosen such that the corresponding quadrature rule is exact for polynomials up to certain degrees. In such cases, error bounds with respect to polynomial degrees can be easily obtained, but limited choices of quadrature points are available. In our recent works \cite{an2022quadrature, an2024bypassing, an2024hyperinterpolation}, we have relaxed this exactness assumption to the Marcinkiewicz–Zygmund (MZ) property \eqref{equ:MZ}, allowing more families of quadrature points to be applied with reasonable error bounds derived. Spherical $t$-designs, as a set of points that leads to quadrature rules with exactness, undoubtedly perform well. Meanwhile, other quadrature points that may not lead to exact quadrature rules also perform reasonably well. Especially, minimal energy points are not suitable for polynomial interpolation since its Lebesgue constant is very large, as observed in An's PhD thesis \cite[Section 3]{an2011distribution}. However, in this paper, we show that minimal energy points are able to solving integral equation by product integration method effectively, which relies on quadrature rules and unfettered hyperinterpolation. Our method and analysis can also be extended to other point sets that satisfy the MZ property \eqref{equ:MZ}, such as randomly distributed points \cite{MR2475947}, or MZ-like properties, such as quasi-Monte Carlo (QMC) designs \cite{MR3246811}. For a comparison between point sets satisfying the MZ property and QMC designs, we refer to our previous paper \cite{an2024bypassing}.

The evaluation of modified moments \eqref{equ:mm} in this paper is based on the Funk--Hecke formula. This formula is effective for certain special singular kernels, such as the algebraic singular kernel $|\x-\y|^{\nu}$ with $\nu<0$, the log kernel $\log|\x-\y|$, and the extended algebraic singular kernel $|\x-\y|^{\nu_1}|\x-\y|^{\nu_2}$ with $\nu_1,\nu_2<0$. For cases involving both singular and oscillatory kernels, we treat the continuous oscillatory kernels as $K(\x,\y)$ and approximate the product of $K$ and the solution $\varphi$ using hyperinterpolation. The degree of hyperinterpolation must be sufficiently high to capture the oscillatory behavior of $K$. The evaluation of modified moments \eqref{equ:mm} is a crucial topic in numerical methods for singular integral equations. Deriving analytical expressions and stable evaluation methods, particularly for mixed singular and oscillatory kernels, will significantly enhance the applicability of our method.

The theory of hyperinterpolation using quadrature points that satisfy the MZ property \eqref{equ:MZ} is a crucial component of our analysis. This theory imposes the $L^2$ assumption on the singular kernel, which is the only additional assumption we have made compared to the common setting of singular integral equations. If the kernel $h$ is too singular to meet this assumption, we expect to implement transformations, such as Atkinson’s transformation \cite{MR2083337} with the choices of ``grading parameters,'' to obtain new integrand that is much smoother and thus fulfills the $L^2$ assumption.

\section*{Acknowledgement}

The work was supported by the National Natural Science Foundation of China (Project No. 12371099). The authors are grateful to Professor Ian H. Sloan for his encouragement and comments.

\bibliographystyle{siamplain}
\bibliography{myref}

\begin{thebibliography}{10}

\bibitem{an2011distribution}
{\sc C.~An}, {\em Distribution of points on the sphere and spherical designs},
  PhD thesis, The Hong Kong Polytechnic University, 2011.

\bibitem{MR2763659}
{\sc C.~An, X.~Chen, I.~H. Sloan, and R.~S. Womersley}, {\em Well conditioned
  spherical designs for integration and interpolation on the two-sphere}, SIAM
  J. Numer. Anal., 48 (2010), pp.~2135--2157,
  \url{https://doi.org/10.1137/100795140}.

\bibitem{an2022quadrature}
{\sc C.~An and H.-N. Wu}, {\em On the quadrature exactness in
  hyperinterpolation}, BIT, 62 (2022), pp.~1899--1919,
  \url{https://doi.org/10.1007/s10543-022-00935-x}.

\bibitem{an2024bypassing}
{\sc C.~An and H.-N. Wu}, {\em Bypassing the quadrature exactness assumption of
  hyperinterpolation on the sphere}, J. Complexity, 80 (2024), p.~101789,
  \url{https://doi.org/10.1016/j.jco.2023.101789}.

\bibitem{an2024hyperinterpolation}
{\sc C.~An and H.-N. Wu}, {\em Is hyperinterpolation efficient in the
  approximation of singular and oscillatory functions?}, J. Approx. Theory, 299
  (2024), p.~106013, \url{https://doi.org/10.1016/j.jat.2023.106013}.

\bibitem{MR0443383}
{\sc P.~M. Anselone}, {\em Collectively compact operator approximation theory
  and applications to integral equations}, Prentice Hall, Englewood Cliffs, NJ,
  1971.

\bibitem{MR0184448}
{\sc P.~M. Anselone and R.~H. Moore}, {\em Approximate solutions of integral
  and operator equations}, J. Math. Anal. Appl., 9 (1964), pp.~268--277,
  \url{https://doi.org/10.1016/0022-247X(64)90042-3}.

\bibitem{MR2083337}
{\sc K.~Atkinson}, {\em Quadrature of singular integrands over surfaces},
  Electron. Trans. Numer. Anal., 17 (2004), pp.~133--150.

\bibitem{MR2934227}
{\sc K.~Atkinson and W.~Han}, {\em Spherical harmonics and approximations on
  the unit sphere: An introduction}, vol.~2044 of Lecture Notes in Mathematics,
  Springer, Heidelberg, 2012, \url{https://doi.org/10.1007/978-3-642-25983-8}.

\bibitem{MR1464941}
{\sc K.~E. Atkinson}, {\em The numerical solution of integral equations of the
  second kind}, Cambridge University Press, Cambridge, 1997.

\bibitem{MR0129147}
{\sc H.~Brakhage}, {\em \"{U}ber die numerische {B}ehandlung von
  {I}ntegralgleichungen nach der {Q}uadraturformelmethode}, Numer. Math., 2
  (1960), pp.~183--196, \url{https://doi.org/10.1007/BF01386221},
  \url{https://doi.org/10.1007/BF01386221}.

\bibitem{MR3246811}
{\sc J.~S. Brauchart, E.~B. Saff, I.~H. Sloan, and R.~S. Womersley}, {\em Q{MC}
  designs: optimal order quasi {M}onte {C}arlo integration schemes on the
  sphere}, Math. Comp., 83 (2014), pp.~2821--2851,
  \url{https://doi.org/10.1090/S0025-5718-2014-02839-1}.

\bibitem{MR3397293}
{\sc D.~Colton and R.~Kress}, {\em Integral equation methods in scattering
  theory}, vol.~72, SIAM, Philadelphia, PA, 2013.

\bibitem{MR2986407}
{\sc D.~Colton and R.~Kress}, {\em Inverse acoustic and electromagnetic
  scattering theory}, vol.~93, Springer, New York, third~ed., 2013.

\bibitem{MR3746524}
{\sc S.~De~Marchi and A.~Kro\'{o}}, {\em Marcinkiewicz-{Z}ygmund type results
  in multivariate domains}, Acta Math. Hungar., 154 (2018), pp.~69--89,
  \url{https://doi.org/10.1007/s10474-017-0769-4}.

\bibitem{delsarte1991geometriae}
{\sc P.~Delsarte, J.-M. Goethals, and J.~J. Seidel}, {\em Spherical codes and
  designs}, Geom. Dedicata, 6 (1977), pp.~363--388,
  \url{https://doi.org/10.1007/bf03187604}.

\bibitem{MR2371991}
{\sc J.~Dick, P.~Kritzer, F.~Y. Kuo, and I.~H. Sloan}, {\em
  Lattice-{N}ystr\"{o}m method for {F}redholm integral equations of the second
  kind with convolution type kernels}, J. Complexity, 23 (2007), pp.~752--772,
  \url{https://doi.org/10.1016/j.jco.2007.03.004}.

\bibitem{driscoll2014chebfun}
{\sc T.~A. Driscoll, N.~Hale, and L.~N. Trefethen}, {\em Chebfun guide}, 2014.

\bibitem{filbir2011marcinkiewicz}
{\sc F.~Filbir and H.~N. Mhaskar}, {\em Marcinkiewicz--{Z}ygmund measures on
  manifolds}, J. Complexity, 27 (2011), pp.~568--596,
  \url{https://doi.org/10.1016/j.jco.2011.03.002}.

\bibitem{graham2002fully}
{\sc I.~G. Graham and I.~H. Sloan}, {\em Fully discrete spectral boundary
  integral methods for {H}elmholtz problems on smooth closed surfaces in
  $\mathbb{{R}}^3$}, Numer. Math., 92 (2002), pp.~289--323,
  \url{https://doi.org/10.1007/s002110100343}.

\bibitem{MR1723850}
{\sc R.~Kress}, {\em Linear integral equations}, vol.~82 of Applied
  Mathematical Sciences, Springer-Verlag, New York, second~ed., 1999,
  \url{https://doi.org/10.1007/978-1-4612-0559-3},
  \url{https://doi.org/10.1007/978-1-4612-0559-3}.

\bibitem{MR2475947}
{\sc Q.~T. Le~Gia and H.~N. Mhaskar}, {\em Localized linear polynomial
  operators and quadrature formulas on the sphere}, SIAM J. Numer. Anal., 47
  (2009), pp.~440--466, \url{https://doi.org/10.1137/060678555}.

\bibitem{MR2582801}
{\sc P.~Leopardi}, {\em Diameter bounds for equal area partitions of the unit
  sphere}, Electron. Trans. Numer. Anal., 35 (2009), pp.~1--16.

\bibitem{Marcinkiewicz1937}
{\sc J.~Marcinkiewicz and A.~Zygmund}, {\em Sur les fonctions indépendantes},
  Fund. Math., 29 (1937), pp.~60--90, \url{http://eudml.org/doc/212925}.

\bibitem{MR2086950}
{\sc H.~N. Mhaskar}, {\em Local quadrature formulas on the sphere}, J.
  Complexity, 20 (2004), pp.~753--772,
  \url{https://doi.org/10.1016/j.jco.2003.06.005}.

\bibitem{mhaskar2001spherical}
{\sc H.~N. Mhaskar, F.~J. Narcowich, and J.~D. Ward}, {\em Spherical
  {M}arcinkiewicz--{Z}ygmund inequalities and positive quadrature}, Math.
  Comp., 70 (2001), pp.~1113--1130,
  \url{https://doi.org/10.1090/S0025-5718-00-01240-0}.

\bibitem{MR1863015}
{\sc H.~N. Mhaskar, F.~J. Narcowich, and J.~D. Ward}, {\em Corrigendum to:
  ``{S}pherical {M}arcinkiewicz--{Z}ygmund inequalities and positive
  quadrature''}, Math. Comp., 71 (2002), pp.~453--454,
  \url{https://doi.org/10.1090/S0025-5718-01-01437-5}.

\bibitem{MR0199449}
{\sc C.~M\"{u}ller}, {\em Spherical harmonics}, Springer Berlin, Heidelberg,
  1966.

\bibitem{MR2457245}
{\sc T.~D. Pham, T.~Tran, and Q.~T. Le~Gia}, {\em Numerical solutions to a
  boundary integral equation with spherical radial basis functions}, ANZIAM J.,
  50 (2008), pp.~C266--C281, \url{https://doi.org/10.1017/S1446181108080322}.

\bibitem{MR1306011}
{\sc E.~A. Rakhmanov, E.~B. Saff, and Y.~M. Zhou}, {\em Minimal discrete energy
  on the sphere}, Math. Res. Lett., 1 (1994), pp.~647--662,
  \url{https://doi.org/10.4310/MRL.1994.v1.n6.a3}.

\bibitem{riesz1918lineare}
{\sc F.~Riesz}, {\em {\"U}ber lineare {F}unktionalgleichungen}, Acta Math., 41
  (1916), pp.~71--98, \url{https://doi.org/10.1007/BF02422940}.

\bibitem{MR638450}
{\sc I.~H. Sloan}, {\em Analysis of general quadrature methods for integral
  equations of the second kind}, Numer. Math., 38 (1981), pp.~263--278,
  \url{https://doi.org/10.1007/BF01397095}.

\bibitem{sloan1995polynomial}
{\sc I.~H. Sloan}, {\em Polynomial interpolation and hyperinterpolation over
  general regions}, J. Approx. Theory, 83 (1995), pp.~238--254,
  \url{https://doi.org/10.1006/jath.1995.1119}.

\bibitem{MR2065291}
{\sc I.~H. Sloan and R.~S. Womersley}, {\em Extremal systems of points and
  numerical integration on the sphere}, Adv. Comput. Math., 21 (2004),
  pp.~107--125, \url{https://doi.org/10.1023/B:ACOM.0000016428.25905.da}.

\bibitem{tran2009boundary}
{\sc T.~Tran, Q.~Le~Gia, I.~H. Sloan, and E.~P. Stephan}, {\em Boundary
  integral equations on the sphere with radial basis functions: error
  analysis}, Applied Numerical Mathematics, 59 (2009), pp.~2857--2871,
  \url{https://doi.org/10.1016/j.apnum.2008.12.033}.

\bibitem{MR1845243}
{\sc R.~S. Womersley and I.~H. Sloan}, {\em How good can polynomial
  interpolation on the sphere be?}, Adv. Comput. Math., 14 (2001),
  pp.~195--226, \url{https://doi.org/10.1023/A:1016630227163}.

\end{thebibliography}
\clearpage

\end{document}